\documentclass[11pt,a4paper]{amsart}
\usepackage{colortbl}

\usepackage{amsthm, amsfonts, amssymb, amsmath, latexsym, enumerate, array}
\usepackage[all]{xy}
\SelectTips{cm}{}

\usepackage{comment}
\usepackage{amssymb,eucal,graphicx}
\usepackage{enumerate}

\numberwithin{equation}{section}
\newtheorem{theo}{THEOREM}[section]
\newtheorem{lemma}[theo]{Lemma}
\newtheorem{cor}[theo]{Corollary}
\newtheorem{prop}[theo]{Proposition}
\newtheorem{dfntn}[theo]{Definition}
\newtheorem{rem}[theo]{Remark}
\newtheorem{claim}[theo]{Claim}

\newtheorem{ass}[theo]{Assumptions}

\newtheorem{ex}[theo]{Example}
\newtheorem{fact}[theo]{FACT}
\newtheorem{facts}[theo]{FACTS}
\newenvironment{rem*}{\begin{rem}\em}{\end{rem}}
\newenvironment{ex*}{\begin{ex}\em}{\end{ex}}
\newenvironment{claim*}{\begin{claim}\em}{\end{claim}}
\newenvironment{facts*}{\begin{facts}\em}{\end{facts}}
\newenvironment{fact*}{\begin{fact}\em}{\end{fact}}

\newcommand{\brref}[1]{(\ref{#1})}

\newcommand{\Pin}[1]{{\mathbb P}^{#1}}
\newcommand{\restrict}[2]{{#1}_{\mid _{#2}}}
\newcommand{\lra}{\longrightarrow}

\newcommand{\Projcal}[1]{\mathbb{ P}({\mathcal #1})}

\newcommand{\oofp}[2]{{\mathcal O}_{\mathbb{ P}^{#1}}({#2})}

\newcommand{\scrollcal}[1]{(\Projcal{#1},\tautcal{#1})}
\newcommand {\xel} {(X, L)}

\newcommand{\tautcal}[1]{{\mathcal O}_{\mathbb{P}({\mathcal#1})}(1)}

\newcommand{\num}{\equiv}

\newcommand{\Pp}{\mathbb P}
\newcommand{\Oc}{\mathcal O}

\newcommand{\FF}{\mathbb{F}}

\newcommand{\Ee}{\mathcal{E}_e}

\usepackage{lineno}

\allowdisplaybreaks[1]

\title{Hilbert schemes of some threefold scrolls over ${\FF}_e$}

\author{Maria Lucia Fania}
\address{Maria Lucia Fania\\ Dipartimento di Ingegneria e Scienze dell'Informazione e Matematica\\
Universit\`{a} degli Studi di L'Aquila\\
Via Vetoio Loc. Coppito\\67100 L'Aquila\\Italy}
\email{fania@univaq.it}

\author{Flaminio Flamini}
\address{Flaminio Flamini\\Dipartimento di Matematica\\ Universit\`a degli Studi di Roma
Tor Vergata \\ Viale della Ricerca Scientifica, 1 - 00133 Roma\\Italy}
\email{flamini@mat.uniroma2.it}

\subjclass[2000]{Primary 14J30, 14J27, 14J60, 14C05; Secondary 14M07, 14N25, 14N30}

\keywords{Ruled varieties, Vector  bundles, Rational surfaces, Hilbert scheme}

\thanks{The authors thank C.\,Ciliberto and E.\,Sernesi for having pointed out questions on Hilbert schemes of threefold scrolls over $\FF_e$, with $e \ge 2$, during the talk of the first author at the Workshop "Algebraic geometry: two days in Rome two", held in Rome in February 2012. The authors are also greateful to the referee for helpful comments and for having posed a question  which allowed us to realize that there was a mistake in the first version of the paper. 
Both authors are members of GNSAGA-INdAM. We acknowledge partial support from MIUR funds, PRIN 2010-2011  project \lq\lq\thinspace Geometria delle Variet\`a Algebriche''.}

\begin{document}



\begin{abstract} Hilbert schemes of suitable smooth, projective threefold scrolls over the Hirzebruch  surface $\FF_e$, $e\ge 2$,  are studied. An irreducible component of the Hilbert scheme parametrizing such varieties is shown to be generically smooth of the expected dimension  and  the general point of such a component is described.
\end{abstract}


\maketitle

\section{Introduction}

Projective varieties are distributed in {\em families}, obtained by suitably varying
the coefficients of their defining equations. The description 
of such families and, in particular, of the properties of their parameter spaces
is a central theme in algebraic geometry. 

Milestones to approach such problems have been both the introduction of  technical tools, like flatness, base change, Hilbert polynomial, etc.,  
and the proof (due to Grothendieck with refinements by Mumford) of the existence of the so called {\em Hilbert scheme}, a closed, 
projective scheme, parametrizing families of projective varieties with suitable constant numerical/projective invariants, 
together with some other fundamental universal properties. 

Since then, Hilbert schemes of projective varieties with given Hilbert polynomial have intere-sted several authors over the years, especially because of the deep connections of the subject with several other important theories in algebraic geometry: zero-dimensional schemes on smooth projective varieties, Brill-Noether theory of line bundles on curves, 
moduli spaces of genus $g$ curves and their stratifications in terms of suitable subvarieties, vector bundles on 
smooth projective varieties,  just to mention a few (for an overview the reader is referred, for instance,  to the bibliography in  \cite{Se}).

For particular cases of projective varieties, one can find in the literature sufficiently 
detailed descriptions of their Hilbert schemes. For example special classes of threefolds in $\Pin{5}$ were studied in \cite{fa-me};  
results for codimension--two projective varieties are due to \cite{e, chang1, chang2};  in codimension three, \cite{kl-mr} considered the case of arithmetically Gorenstein closed subschemes in a projective space, whereas 
\cite{kmmnp} dealt with determinantal schemes. For codimension greater than or equal to  two, Hilbert schemes of {\em Palatini scrolls} in $\Pin{n}$, with $n$ odd, 
have been treated in \cite{fae-fan} while in  \cite{fae-fan2} Hilbert schemes of varieties defined by maximal minors were considered.  We also mention results in \cite{kl}
concerning Hilbert schemes of determinantal schemes.

An important class of projective varieties is that of $r$-{\em scrolls} in ${\Pp^n}$, namely 
ruled varieties over a smooth base which are embedded in ${\Pp^n}$ in such a way that the rulings are $r$-dimensional linear subspaces of ${\Pp^n}$. This class is important not only because it usually comes out as a fundamental special case from problems in 
classical adjunction theory (cf.\;e.g.\,\cite{BESO,ot1}), but mainly because it is strictly related to the 
study of vector bundles of rank $(r+1)$ over smooth projective varieties. 

For rank-two, degree $d$ vector bundles over genus $g$ curves (equivalently, surface scrolls of degree $d$ and 
sectional genus $g$), apart from the classical approach of C. Segre (\cite{Seg}) and of  some other more recent 
partial results as, for instance,  in  \cite{Ghio,APS,GP1,GP3}, a systematic study of Hilbert schemes of such surface scrolls has been developed in the series of papers \cite{calabri-ciliberto-flamini-miranda:non-special, calabri-ciliberto-flamini-miranda:non-special1, calabri-ciliberto-flamini-miranda:special, ciliberto-flamini:special1}, where the authors 
bridged the Hilbert scheme approach with the vector-bundle one, showing in particular how projective geometry and degeneration techniques 
can be used in order to improve some known results about rank-two vector bundles on curves and also to obtain some new ones.

A similar approach has been used to study Hilbert schemes of $r$-scrolls, $r \ge 1$, over smooth projective surfaces $S$, with $S$ either a $K3$ (\cite{Fla}) or the Hirzebruch surfaces $\FF_0$ and $\FF_1$ (\cite{be-fa,be-fa-fl}). In the authors' opinion, it would be interesting to develop the use of projective geometry and of degeneration techniques in order to study possible limits of vector-bundles, of any rank, on classes of smooth, projective varieties.

In this paper we focus on some classes of $1$--scrolls over Hirzebruch surfaces $\FF_e$, with $e \ge 2$. 
Rank--two vector bundles on Hirzebruch surfaces are classified in \cite{bro}; some of their cohomological and ampleness properties are studied in \cite{al-be}; moduli spaces of rank-two vector bundles on Hirzebruch surfaces are considered, for example,  in \cite{ApBr}. 
On the other hand, very little is known about Hilbert schemes of $1$--scrolls over $\FF_e$. 

We consider vector bundles $\mathcal {E}_e$ arising as extensions of suitable line bundles over $\FF_e$ and with Chern classes $c_1(\mathcal E_e) = 3 C_e + b_e f$, $c_2(\mathcal E_e) = k_e$, where $C_e$ and $f$ are respectively the section of minimal self-intersection and a fiber of $\FF_e$, whereas $b_e$ and $k_e$ are integers suitably chosen (cf.\,Assumptions \ref{ass:AB}, \ref{ass:AB2}). Such a choice of $c_1(\mathcal E_e) = 3 C_e + b_e f$ and of the integers $b_e, k_e$   gives the first case for which the bundle $\mathcal E_e$ is both 
{\em uniform} and {\em very-ample}  (cf.\S \,\ref{S:scrollsFe} and Remark \ref{rem:added3}).

Let therefore $X_e$ be a threefold in ${\Pp^{n_e}}$ which is a scroll over $\FF_e$,  ${n_e} \ge 6$, $e\ge 2$, that is $X_e\cong \Projcal{E_e}$ is the projectivization of a rank--two vector bundle $ \mathcal{E}_{e}$ over $\FF_e$ as above. 
We assume  ${n_e} \ge 6$ because  it is known that there are no such scrolls when ${n_e} \le 5$, see \cite{ot1}.

If one wants to parametrize varieties $X_e$ of this type,  the first tasks to be tackled are:  

\vskip 8pt

\noindent
(i) looking at $[X_e]$ as a point of a component of $\mathcal H_3^{d_e,n_e}$, the  Hilbert scheme parametrizing $3$-dimensional subvarieties of $\Pp^{n_e}$  of degree $d_e$ having same Hilbert polynomial $P_{X_e} (T)$ as that of $X_e$, and

\vskip 8pt

\noindent
(ii) understanding the general point of such a component in $\mathcal H_3^{d_e,n_e}$.

\vskip 8pt

For  $e=0, 1$, the above problems have been considered  in \cite{be-fa,be-fa-fl}, where the Hilbert schemes of  threefold scrolls $X_{0}$ and $X_{1}$ were studied. 
Namely, it was proved that the irreducible component containing such scrolls is generically smooth, of the expected dimension, and its general point is actually a threefold scroll, that is  the component is filled up by scrolls. The aim of this paper is to see what happens if the base of the scroll is  $\FF_e$, with $e\ge 2$. 

Our main results, \,Theorems \ref{thm:parcount},\, and \ref{thm:puntogenerico}, in particular  answer  a question on Hilbert schemes of threefold scrolls over $\FF_e$, $e \ge 2$, pointed out to us by C. Ciliberto and E. Sernesi and for which we thank them.

In this paper, we prove that there exists an irreducible component $\mathcal X_e$ of $\mathcal H_3^{d_e,n_e}$, containing such scrolls, which is generically smooth, of the {\em expected dimension} and such that $[X_e]$ belongs to the smooth locus of $\mathcal X_e$ (cf. Theorem \ref{thm:HilbertFe}). In contrast with the $e=0, 1$ cases, we show that the family of constructed scrolls $X_e$'s surprisingly does not fill up the component $\mathcal X_e$ (cf. Theorem \ref{thm:parcount}).  

We thus exhibit a smooth variety $X_{\epsilon} \subset \Pp^{n_e}$, which is a candidate to represent the general point of $\mathcal X_e$. 
More precisely, we show that $X_{\epsilon}$ corresponds to the general point of an irreducible component, of the same Hilbert scheme 
$\mathcal H_3^{d_e,n_e}$, which is generically smooth and of the expected dimension. We then show that $X_{\epsilon}$ flatly degenerates in $\Pp^{n_e}$ to  a general threefold scroll $X_e$ as above, in such a way that the base--scheme of the flat, embedded degeneration is entirely contained in 
$\mathcal X_e$. By the generic smoothness of $\mathcal X_e$, we can conclude that $X_{\epsilon}$ is actually the general point of $\mathcal X_e$ (cf.\,\S's\,\ref{S:5.3new}, \ref{S:genpointeps}).  

The paper is structured in the following way. In Section \ref{notation} notation is fixed.
In Section \ref{S:vbFe}, following \cite{be-fa,be-fa-fl}, we consider suitable rank-two vector bundles over $\FF_e$, with  $e \ge 2$. In Section \ref{S:scrollsFe} we consider  Hilbert schemes parametrizing families of $3$-dimensional scrolls over $\FF_e$, $e \ge 2$. In Section \ref{S:genpoint} a description of  the general point of the component $\mathcal X_e$ determined in Theorem \ref{thm:HilbertFe} is presented. More precisely,  in  \S\;\ref{S:5.3new} we first construct the candidate $X_{\epsilon}$ and analyze some of its properties, similar to those investigated for $X_e$ in Sections\;\ref{S:vbFe}, \ref{S:scrollsFe}; then, in  \S\;\ref{S:genpointeps}, we show that $X_{\epsilon}$ actually corresponds to the general point of  $\mathcal X_e$. Finally, Section \ref{Examples} contains some concrete examples of Hilbert scheme of scrolls over some $\FF_e$, with $e \ge 2$ and $e$ both even and odd.


\section{Notation and Preliminaries}
\label{notation}
The following notation will be used throughout this work.
  \begin{enumerate}
\item [ ] $X$ is a smooth, irreducible, projective variety of dimension $3$ (or simply a threefold);
\item[ ]$\chi(\mathcal F) = \sum(-1)^ i h^i(\mathcal F)$, the Euler characteristic of $\mathcal F$, where $\mathcal F$ is any vector bundle of rank $r \geq 1$ on $X$;
\item[ ] $c_i(\mathcal F)$,  the $i$-th Chern class
of $\mathcal F$;
\item[ ]  $\restrict{\mathcal F}{Y}$ the restriction of $\mathcal F$ to a subvariety $Y;$
\item[ ]  $K_X$ the canonical bundle of $X.$ When the context is clear, $X$ may be dropped, so $K_X = K$;
\item[ ] $c_i = c_i(X)$,  the $i$-th Chern class
of $X$;
\item[ ] $d = \deg{X} = L^3$, the degree of $X$ in the embedding
given by a very-ample line bundle $L$;
\item[] $g = g(X),$ the sectional genus of $\xel$ defined by
$2g-2=(K+2L)L^2;$
\item[] if $S$ is a smooth surface, $\equiv$ will denote the numerical equivalence of divisors on $S$. 
\end{enumerate}

For non-reminded terminology and notation, we basically follow \cite{H}. 

\begin{dfntn}\label{specialvar} A pair $(X, L)$, where
$L$ is an ample line bundle on a threefold $X,$ is
a {\it scroll} over a normal variety $Y$ if there exist an ample line
bundle $M$ on $Y$ and a surjective morphism  $\varphi: X \to Y$ with
connected fibers
such that $K_X + (4 - \dim Y) L = \varphi^*(M).$
\end{dfntn}

In particular, if $Y$ is smooth and $\xel$ is a scroll over $Y$,  then (see \cite[Prop. 14.1.3]{BESO})
$X \cong \Projcal{E} $, where ${\mathcal E}= \varphi_{*}(L)$ and
$L$ is the tautological  line bundle on $\Projcal{E}.$ Moreover, if $S \in |L|$ is a smooth divisor,
then (see e.g. \cite[Thm. 11.1.2]{BESO}) $S$ is the blow
up of $Y$ at $c_2(\mathcal{E})$ points; therefore $\chi({\mathcal O}_{Y}) = \chi({\mathcal O}_{S})$ and
\begin{equation}\label{eq:d}
d : = L^3 = c_1^2(\mathcal{E})-c_2(\mathcal{E}).
\end{equation}

Throughout this work, the scroll's base $Y $ will be the Hirzebruch surface $\FF_e= \Pp(\Oc_{\Pp^1} \oplus\Oc_{\Pp^1}(-e))$, 
with $e \ge 0$ an integer. 

Let $\pi_e : \FF_e \to \Pp^1$ be the natural projection onto the base. Then 
${\rm Num}(\FF_e) = \mathbb{Z}[C_e] \oplus \mathbb{Z}[f],$ where:

\noindent
$\bullet$ $C_e$ denotes the unique section corresponding to the morphism 
$\Oc_{\Pp^1} \oplus\Oc_{\Pp^1}(-e) \to\!\!\!\to \Oc_{\Pp^1}(-e)$ on $\Pp^1$, and

\noindent
$\bullet$ $f = \pi^*(p)$, for any $p \in \Pp^1$.

\noindent
In particular$$C_e^2 = - e, \; f^2 = 0, \; C_ef = 1.$$

Let $\mathcal{E}_e$ be a rank-two vector bundle over $\FF_e$ and let $c_i(\mathcal{E}_e)$ be its  $i^{th}$-Chern class. Then $c_1( \mathcal{E}_e) \num a C_e + b f$, for some $ a, b \in \mathbb Z$, and $c_2(\mathcal{E}_e) \in \mathbb Z$.

%
%

\section{Some rank-two vector bundles over $\FF_e$, for $e \ge 2$} \label{S:vbFe}

In \cite{be-fa,be-fa-fl} the authors considered suitable rank-two vector bundles over $\FF_e$, for $e =0,1$. In this and the following section, 
we will focus on the case $e \ge 2$. Therefore, unless otherwise stated, from now on we will use the following:

\begin{ass}\label{ass:AB} Let $e \ge 2$, $b_e$, $k_e$ be integers. Let 
${\mathcal E}_e$ be a rank-two vector bundle over $\FF_e$, with
$$c_1({\mathcal E}_e) \num 3 C_e + b_e f \;\; {\rm and} \;\; c_2({\mathcal E}_e) = k_e,$$such that 

\begin{itemize}
\item[(i)] $h^0( \mathcal E_e)\ge 7$
\item[(ii)] $b_e\ge 3e+1$
\item[(iii)] $k_e+e >b_e$
\end{itemize} (cf. \S\,\ref{S:scrollsFe} below and \cite[Prop.7.2]{al-be},  for motivation). 
Moreover, there exists an exact sequence
\begin{equation}\label{eq:al-be}
0 \to A_e \to {\mathcal E}_e \to B_e \to 0,
\end{equation}where $A_e$ and $B_e$ are line bundles on $\FF_e$ such that
\begin{equation}\label{eq:al-be3}
A_e \num 2 C_e + (2b_e-k_e-2e) f \;\; {\rm and} \;\; B_e \num C_e + (k_e - b_e + 2e) f
\end{equation}(cf. \cite[Prop.7.2]{al-be} and \cite{bro}). 
\end{ass} From \eqref{eq:al-be}, in particular, one has $c_1({\mathcal E}_e) = A_e + B_e \;\; {\rm and} \;\; c_2({\mathcal E}_e) = A_eB_e$.
\vspace{3mm}

Exact sequence \eqref{eq:al-be} gives important preliminary information on the cohomology of ${\mathcal E}_e$, $A_e$ and $B_e$. Indeed, one has

\begin{lemma}\label{lemma:comEAB} With  Assumptions \ref{ass:AB}, one has
$$
h^j({\mathcal E}_e) = h^j(A_e) = 0, \;\mbox{for}\; j \geq 2, \;\; {\rm } \;\; h^i(B_e) = 0, \;\mbox{for}\; i \geq 1,$$
$$h^0(A_e) =  6b_e-3k_e-9e+3 + h^1(A_e),\;\; h^0(B_e) = 2k_e-2b_e+3e+2$$ and 
\begin{equation}\label{eq:h0E}
h^0({\mathcal E}_e) =  4b_e-k_e-6e+5 + h^1({\mathcal E}_e).
\end{equation}
\end{lemma}

\begin{proof} For dimension reasons, it is clear that
$h^j({\mathcal E}_e) = h^j(\FF_e, A_e) = h^j(\FF_e, B_e) = 0, \; j \geq 3$. 

By Serre duality on $\FF_e$, $$h^2(A_e) = h^0( - 4C_e - (2b_e-k_e-e +2)f) = 0 \; {\rm and}  
\; h^2(B_e) = h^0( - 3 C_e - (k_e - b_e + 3e+2)f) = 0,$$since $K_{\FF_e} \equiv - 2 C_e - (e+2) f$. 
In particular, this implies that also $h^2({\mathcal E}_e) = 0$.  

We claim that, under Assumptions \ref{ass:AB}, we also have $h^1(B_e) = 0$. Indeed,
since $B_e \num C_e + (k_e - b_e + 2e) f$, it follows that  $R^{1}\pi_{*}(B_e)=0$ and thus by Leray's isomorphism, 
\begin{eqnarray*}
h^1(B_e) &=& h^1(\Pp^1, (\Oc_{\Pp^1}\oplus \Oc_{\Pp^1}(-e))\otimes \Oc_{\Pp^1}(k_e-b_e+2e)) \\\nonumber
&=&h^1(\Pp^1, \Oc_{\Pp^1}(k_e-b_e+2e))+h^1(\Pp^1, \Oc_{\Pp^1}(k_e-b_e+e))=0,
\end{eqnarray*}
by Assumptions \ref{ass:AB}-(iii). 

Thus we have
\begin{equation}\label{eq:comseq2}
\chi(A_e) = h^0(A_e) - h^1(A_e), \;\; \chi(B_e) =h^0(B_e), \;\; \chi({\mathcal E}_e) = h^0({\mathcal E}_e) - h^1({\mathcal E}_e).
\end{equation}From the Riemann-Roch formula, we have
$$\chi(A_e) = \frac{1}{2}A_e (A_e-K_{\FF_e}) + 1 = $$
$$\frac{1}{2} \left(2C_e + (2b_e-k_e-2e) f \right) \left(4C_e + (2b_e-k_e-e+2)f\right) + 1 = 6b_e-3k_e-9e+3,$$whereas
$$\chi(B_e) = h^0(B_e) =  \frac{1}{2} B_e (B_e-K_{\FF_e}) + 1 =  $$
$$\frac{1}{2} \left(C_e + (k_e-b_e+2e) f \right) \left(3C_e + (k_e-b_e+3e+2)f\right) + 1 = 2k_e-2b_e+3e+2.$$Since $\chi({\mathcal E}_e) =
\chi(A_e) + \chi(B_e)$, the remaining statements follow from the cohomology sequence associated with \eqref{eq:al-be} and from \eqref{eq:comseq2}.
\end{proof}

\medskip
From Lemma \ref{lemma:comEAB}  we have:
\begin{equation}\label{eq:comseq}
0 \to H^0(A_e) \to H^0({\mathcal E}_e) \to H^0(B_e) \stackrel{\partial}{\longrightarrow} H^1(A_e) \to H^1({\mathcal E}_e) \to 0,
\end{equation}where $\partial$ is the {\em coboundary map} determined by the extension \eqref{eq:al-be}. Thus
\begin{equation}\label{eq:h1Eh1A}
h^1({\mathcal E}_e) \leq h^1(A_e).
\end{equation}

\begin{rem}\label{rem:equivalence} \normalfont{From \eqref{eq:h0E}, Assumption \ref{ass:AB}(i) 
is equivalent to $4b_e-k_e-6e+5 + h^1({\mathcal E}_e)\ge 7$, that is $k_e\le 4b_e-6e-2 + h^1({\mathcal E}_e)$.
}
\end{rem}

\vspace{2mm}

%
%

\subsection{Vector bundles in ${\rm Ext}^1(B_e,A_e)$}\label{S:vb} This subsection is devoted to an analysis of
vector bundles fitting in the exact sequence \eqref{eq:al-be}. We need the following:

\begin{lemma}\label{lem:ext1} With Assumptions \ref{ass:AB}, one has
{\small
\begin{equation}\label{eq:dimExt1}
\dim({\rm Ext}^1(B_e,A_e)) = \left\{
\begin{array}{ccc}
0 & & {\rm for} \; b_e-e < k_e <  \frac{3b_e+2-5e}{2}\\
 & & \\
5e+2k_e-3b_e-1 & & {\rm for} \;\frac{3b_e+2-5e}{2} \le k_e <  \frac{3b_e+2-4e}{2}\\
 & & \\
9e+4k_e-6b_e-2 & &  {\rm for} \;\frac{3b_e+2-4e}{2} \le k_e \le 4b_e - 6e-2+h^1(\Ee).
\end{array}
\right.
\end{equation}
}

\end{lemma}

\begin{proof} By standard facts, ${\rm Ext}^1(B_e,A_e) \cong H^1(A_e-B_e)$. From \eqref{eq:al-be3}, 
\begin{equation}\label{eq:A-B} 
A_e - B_e \equiv C_e + (3b_e-2k_e-4e)f.
\end{equation}
Now $R^{i}{\pi_e}_*(C_e + (3b_e-2k_e-4e)f)=0$, for $i>0$,  and ${\pi_e}_*(C_e + (3b_e-2k_e-4e)f) \cong (\Oc_{\Pp^1} \oplus \Oc_{\Pp^1} (-e)) \otimes \Oc_{\Pp^1} (3b_e-2k_e-4e)$, 
hence, from Leray's isomorphism we have
\begin{eqnarray*}
h^1(A_e - B_e)&=& h^1(\Pp^1, (\Oc_{\Pp^1} \oplus \Oc_{\Pp^1} (-e)) \otimes \Oc_{\Pp^1} (3b_e-2k_e-4e))\\
&=& h^1(\Oc_{\Pp^1} (3b_e-2k_e-4e)) + h^1(\Oc_{\Pp^1} (3b_e-2k_e-5e))
\end{eqnarray*}By Serre's duality on $\Pp^1$, the previous sum coincides with
$$h^0(\Oc_{\Pp^1} (2k_e+4e-3b_e-2)) + h^0(\Oc_{\Pp^1} (2k_e+5e-3b_e-2)).$$Put 
$\alpha := 2k_e+4e-3b_e-2$ and $\beta := 2k_e+5e-3b_e-2$; note that $\beta=\alpha +e$.

\smallskip

\noindent
$\bullet$ If $\beta<0$ then also $\alpha <0$ and thus $h^1(A_e-B_e)=0.$

\smallskip

\noindent
$\bullet$ If $\beta \ge0$ and $\alpha <0$ then $h^1(A_e-B_e)=\beta+1.$

\smallskip

\noindent
$\bullet$ Finally, if $\alpha \ge0$ then $\beta >0$ and thus $h^1(A_e-B_e)=\alpha+\beta+2.$

Now observe that 
$$\beta<0 \Leftrightarrow k_e <  \frac{3b_e+2-5e}{2} \;\; \mbox{and} \;\; \alpha <0 \Leftrightarrow k_e <  \frac{3b_e+2-4e}{2}.$$Moreover, since 
$e \ge 2$, by Assumptions \ref{ass:AB}-(ii) one easily verifies that 
all such numerical conditions are compatible with Assumptions \ref{ass:AB}-(i) and (iii) (cf. also Rem.\,\ref{rem:equivalence}), 
in other words one has $$b_e - e < \frac{3b_e+2-5e}{2} < \frac{3b_e+2-4e}{2} < 4b_e-6e-2 \le 
4b_e-6e-2 + h^1({\mathcal E}_e).$$Hence \eqref{eq:dimExt1} follows.
\end{proof}

\begin{cor}\label{cor:dimExt1}  With Assumptions \ref{ass:AB}, for 
$b_e-e < k_e < \frac{3b_e+2-5e}{2}$, one has ${\mathcal E}_e = A_e \oplus B_e$. 
\end{cor}

In \S\,\ref{S:genpoint} (cf. the proof of Theorem \ref{thm:parcount}), we shall also need to know $\dim(Aut({\mathcal E}_e)) = h^0({\mathcal E}_e \otimes \mathcal E_e^{\vee})$.

\begin{lemma}\label{lem:autE} With Assumptions \ref{ass:AB}, take any $\Ee \in {\rm Ext}^1(A_e,B_e)$. Then: 
{\footnotesize
\begin{equation}\label{eq:autE}
h^0({\mathcal E}_e \otimes {\mathcal E}_e^{\vee}) = \left\{
\begin{array}{ccl}
6b_e-4k_e -9e+4 & & {\rm for} \;\; b_e-e < k_e < \frac{3b_e+2-5e}{2}\\
 & & \\
3b_e-2k_e -4e+2 & &  {\rm for} \;\; \frac{3b_e+2-5e}{2} \leq k_e \le  \frac{3b_e-4e}{2} \;\; {\rm and} \;\; {\mathcal E}_e \;\; {\rm general}\\
& & \\
1 & &  {\rm for} \;\; \frac{3b_e-4e}{2} < k_e \le 4b_e - 6e-2 + h^1(\Ee) \;\; {\rm and} \;\; {\mathcal E}_e \; {\rm general}.
\end{array}
\right.
\end{equation}
}
\end{lemma}
\begin{proof} (i) According to Corollary \ref{cor:dimExt1}, for $b_e-e < k_e < \frac{3b_e+2-5e}{2}$, 
${\mathcal E}_e = A_e \oplus B_e$. Therefore
$${\mathcal E}_e \otimes {\mathcal E}_e^{\vee} \cong \Oc_{\FF_e}^{\oplus 2} \oplus (A_e-B_e) \oplus (B_e-A_e).$$ From \eqref{eq:al-be3},
\begin{equation}\label{eq:B-A}
B_e-A_e \equiv - C_e + (2k_e-3b_e+4e) f,
\end{equation}so it is not effective, since it negatively intersects the irreducible, moving curve $f$.

From \eqref{eq:A-B} and from the proof of Lemma \ref{lem:ext1}, one has 
$$h^0(A_e-B_e) = h^0(C_e + (3b_e-2k_e-4e) f) = h^0(\Oc_{\Pp^1} (3b_e-2k_e-4e)) + h^0(\Oc_{\Pp^1} (3b_e-2k_e-5e)).$$Put 
$\alpha':= 3b_e-2k_e-4e$ and $\beta' := 3b_e-2k_e-5e$; note that $\beta'=\alpha'-e$

Since $k_e <  \frac{3b_e-5e+2}{2}$, $\Oc_{\Pp^1} (3b_e-2k_e-4e)$ is always effective whereas $\Oc_{\Pp^1} (3b_e-2k_e-5e)$ is effective unless  
$3b_e-2k_e-5e = -1$. So 
$h^0(\Oc_{\Pp^1} (3b_e-2k_e-4e)) + h^0(\Oc_{\Pp^1} (3b_e-2k_e-5e)) = 6b_e-4k_e-9e+2$; taking into account also
$h^0(\Oc_{\FF_e}^{\oplus 2})$, we conclude in this case.

\vskip 5pt

\noindent
(ii)-(iii)  We treat here the remaining cases in \eqref{eq:autE}. Recall that the upper-bound $k_e \le 4b_e-6e-2 + h^1(\Ee)$ comes from 
Assumptions \ref{ass:AB}-(i) (cf. Remark \ref{rem:equivalence}).  

According to Lemma \ref{lem:ext1}, when $k_e \ge \frac{3b_e+2-5e}{2}$, one has $\dim({\rm Ext}^1(B_e,A_e)) >0$. Therefore, let 
${\mathcal E}_e\in  {\rm Ext}^1(B_e,A_e)$ be general. Using the fact that ${\mathcal E}_e$ is of rank two and fits in the exact sequence \eqref{eq:al-be}, we have
$${\mathcal E}_e^{\vee} \cong {\mathcal E}_e \otimes \Oc(-A_e-B_e),$$since $c_1({\mathcal E}_e) = A_e+B_e$. Tensoring \eqref{eq:al-be} respectively 
by ${\mathcal E}_e^{\vee}$,  $-B_e$, $-A_e$, we get the following exact diagram
\begin{equation}\label{eq:diag}
\begin{array}{rcccccc}
      &  0 &     &    0        &     & 0 & \\

      & \downarrow  &     &     \downarrow       &     & \downarrow & \\

 0 \to   &  A_e-B_e &    \to  & {\mathcal E}_e (-B_e) &   \to  & \Oc_{\FF_e} & \to 0 \\

      &  \downarrow &     &       \downarrow     &     & \downarrow & \\

 \;\;\;0 \to &  {\mathcal E}_e (-B_e) & \to & {\mathcal E}_e \otimes {\mathcal E}_e^{\vee} & \to & {\mathcal E}_e (-A_e) & \to 0 \\

     &   \downarrow &     &       \downarrow     &     & \downarrow & \\

    0 \to  &  \Oc_{\FF_e} &      \to &     {\mathcal E}_e (-A_e) & \longrightarrow   &B_e -A_e & \to 0 \\

          &     \downarrow &     &     \downarrow       &     & \downarrow & \\

          &  0 &     &           0 &     & 0 & 

\end{array}
\end{equation}
We want to compute both $h^0({\mathcal E}_e (-B_e))$ and $h^0({\mathcal E}_e (-A_e))$.

From the cohomology sequence associated to the first row of diagram (\ref{eq:diag}) we get
$$0\to H^0(A_e-B_e)\to H^0({\mathcal E}_e (-B_e))\to H^0(\Oc_{\FF_e}) \stackrel{\widehat{\partial}}{\longrightarrow} H^1(A_e-B_e).$$
Observe that the coboundary map $$H^0(\Oc_{\FF_e}) \stackrel{\widehat{\partial}}{\longrightarrow} H^1(A_e-B_e),$$ has to be injective since it corresponds to the choice of the non-trivial extension class
$\eta_{\Ee} \in {\rm Ext}^1(B_e,A_e)$ associated to ${\mathcal E}_e$ general. Thus 
$$h^0({\mathcal E}_e (-B_e))=h^0(A_e -B_e))=h^0(\Oc_{\Pp^1} (\alpha')) + h^0(\Oc_{\Pp^1} (\beta')),$$with $\alpha'$ and $\beta'$ as in Case (i) above. 

Since $k_e\ge  \frac{3b_e+2-5e}{2} $, then $\beta'\le -2$ hence  $h^0(\Oc_{\Pp^1} (\beta'))=0$. Thus, $h^0( {\mathcal E}_e (-B_e)) = h^0(\Oc_{\Pp^1} (\alpha'))$.  Morover, $h^0(\Oc_{\Pp^1} (\alpha'))=0$ if and only if $k_e> \frac{3b_e-4e}{2}$; thus 

{\small
\begin{equation}
\label{hOalpha}
 h^0( {\mathcal E}_e (-B_e))=\left\{
\begin{array}{ccl}
3b_e-2k_e -4e+1 & &  {\rm for} \; \frac{3b_e+2-5e}{2} \leq k_e \le \frac{3b_e-4e}{2}\\
0 & &  {\rm for} \; k_e> \frac{3b_e-4e}{2}
\end{array}
\right.
\end{equation}
}

From the third row of diagram \eqref{eq:diag}, since $B_e -A_e$ is not effective (cf. \eqref{eq:B-A}), it follows that 
$h^0( {\mathcal E}_e (-A_e))=h^0(\Oc_{\FF_e})=1$, 
thus $H^0( {\mathcal E}_e (-A_e))\cong {\mathbb C}$.

From the second column of diagram \eqref{eq:diag}, we have
$$0\to  H^0({\mathcal E}_e (-B_e))\to H^0( {\mathcal E}_e \otimes {\mathcal E}_e^{\vee})\stackrel{\psi}{\longrightarrow} H^0({\mathcal E}_e (-A_e)) \cong \mathbb{C}  \to H^1({\mathcal E}_e (-B_e))\to \cdots .$$

\begin{claim}\label{cl:fi}
The map $\psi$ is surjective. 
\end{claim}
\begin{proof}[Proof of Claim \ref{cl:fi}] From the first two columns of diagram \eqref{eq:diag} and the fact that the coboundary map ${\widehat{\partial} }$ is injective, as remarked above, we have
\[
\begin{array}{rcccccc}
   & &  0  &  &    & H^0({\mathcal E}_e \otimes {\mathcal E}_e^{\vee}) & \\

      &   & \downarrow    &       &     & \downarrow^{\psi}  & \\

& 0 \to &    H^0(\Oc_{\FF_e})\ & \stackrel{\cong}{\longrightarrow}  & &H^0( {\mathcal E}_e (-A_e)) & \to 0 \\

     &    &            \downarrow^{\widehat{\partial}}     &  &   & \downarrow^{\tilde{\partial} }& \\

     &  &           H^1(A_e-B_e) & \longrightarrow &   &H^1( {\mathcal E}_e (-B_e)) &  \\
\end{array}
\]Since $H^0({\mathcal E}_e (-A_e))) \cong \mathbb{C}$, $\psi$ is not surjective iff  $\psi \num 0$, which is equivalent 
to ${\tilde{\partial} }$ injective and this is impossible since, from the first column of diagram \eqref{eq:diag},  we have
$$H^0(\Oc_{\FF_e})  \stackrel{\widehat{\partial}}{\longrightarrow} H^1(A_e-B_e)  \to H^1( {\mathcal E}_e (-B_e))$$
and the composition of the above two maps is ${\tilde{\partial} }$. This proves the claim. 
\end{proof}

From Claim \ref{cl:fi},  we conclude that
\begin{equation}\label{endom}
h^0( {\mathcal E}_e \otimes {\mathcal E}_e^{\vee})=h^0({\mathcal E}_e (-B_e))+1. 
\end{equation}Combining \eqref{hOalpha} and \eqref{endom} we determine 
$h^0( {\mathcal E}_e \otimes {\mathcal E}_e^{\vee})$ in the case  ${\mathcal E}_e \in {\rm Ext}^1(B_e,A_e)$ is general.
\end{proof}

\begin{rem}\label{rem:endom} {\normalfont (1) Note that when 
$\frac{3b_e+2-5e}{2} \leq k_e \le \frac{3b_e-4e}{2}$ (which makes sense only for $e \ge 2$), any $\Ee \in {\rm Ext}^1(A_e,B_e)$ is 
such that $h^0( {\mathcal E}_e \otimes {\mathcal E}_e^{\vee})>1$, that is  $ {\mathcal E}_e$  is not simple. This gives a different situation with respect to 
cases $e = 0,1$. Indeed, for $e =1$, $b_1 \ge 4$, when $\dim({\rm Ext}^1(B_1,A_1))>0$, 
$\mathcal E_1 \in  {\rm Ext}^1(B_1,A_1)$ general is always simple (cf. \cite[Lemmas 3.4, 3.6]{be-fa-fl}). 
Similar computations hold for the case $e=0$ (cf. \eqref{eq:aut01} below).  

 \vskip 5pt 

\noindent
(2) When $h^0( {\mathcal E}_e \otimes {\mathcal E}_e^{\vee})=1$ (from 
\eqref{eq:autE} this, for instance, happens when ${\mathcal E}_e \in {\rm Ext}^1(B_e,A_e)$ is general with $\frac{3b_e-4e}{2} < k_e \le 
4b_e - 6e + 2 + h^1(\Ee)$), ${\mathcal E}_e$ has to be necessarily indecomposable.
} 
\end{rem}

%
%

\subsection{Non-special bundles ${\mathcal E}_e$}\label{S:h1E} For our analysis in \,\S\,\ref{S:scrollsFe}, it 
is fundamental to deal with vector bundles $\mathcal E_e$ with no higher cohomology, in particular {\em non-special} that is with $h^1(\mathcal E_e) = 0$. Indeed, if  $\mathcal E_e$ turns out to be very-ample, the fact that $\mathcal E_e$ has no higher cohomology not only implies that the ruled threefold $\Pp(\mathcal E_e)$ isomorphically embeds via the tautological linear system as a smooth, linearly normal scroll $X_e$ in the projective space $\Pp^{n_e}$ of (the {\em expected}) dimension $n_e := h^0(\mathcal E_e)-1$, but mainly its non-speciality ensures good behavior of the Hilbert point $[X_e]$ in its Hilbert scheme (cf. proof of Claim\,\ref{cl:clar}).

From Lemma \ref{lemma:comEAB}, having $\mathcal E_e$ with no higher cohomology is equivalent to having $\mathcal E_e$ non-special. In this subsection, we therefore find sufficient conditions  for the non-speciality of $\mathcal E_e$, coming from \eqref{eq:h1Eh1A} and the cohomology of $A_e$.

%
%

\begin{lemma}\label{lem:h1A} With Assumptions \ref{ass:AB}, one has
{\small
\begin{equation}\label{eq:h1A}
h^1(A_e) = \left\{
\begin{array}{ccc}
0 & & {\rm for} \;\; b_e-e < k_e < 2b_e+2-4e \\
 & & \\
4e+k_e-2b_e-1 & & {\rm for} \;\; 2b_e+2-4e\le k_e < 2b_e+2-3e \\
 & & \\
7e+2k_e-4b_e-2 & & {\rm for} \;\; 2b_e+2-3e\le k_e < 2b_e+2-2e \\
 & & \\
9e+3k_e-6b_e-3 & &  {\rm for} \;\; 2b_e+2-2e \le k_e \le 4b_e-6e-2 + h^1(\Ee).
\end{array}
\right.
\end{equation}
}
\end{lemma}
\begin{proof}  Fom \eqref{eq:al-be3} ${\pi_e}_*(A_e) \cong Sym^2(\Oc_{\Pp^1}\oplus \Oc_{\Pp^1}(-e))\otimes \Oc_{\Pp^1}(2b_e-k_e-2e)$ and $R^{i}{\pi_e}_*(A_e)=0$ for $i>0$. Hence by Leray's isomorphism, 
{\small 
 \begin{eqnarray*}
 h^1(A_e)&= &h^1(Sym^2(\Oc_{\Pp^1}\oplus \Oc_{\Pp^1}(-e))\otimes \Oc_{\Pp^1}(2b_e-k_e-2e))\\ \nonumber
 &=& h^1((\Oc_{\Pp^1}\oplus \Oc_{\Pp^1}(-e)\oplus  \Oc_{\Pp^1}(-2e))\otimes \Oc_{\Pp^1}(2b_e-k_e-2e))\\
  &= &h^1(\Oc_{\Pp^1}(2b_e-k_e-2e))+h^1(\Oc_{\Pp^1}(2b_e-k_e-3e))+h^1(\Oc_{\Pp^1}(2b_e-k_e-4e))
  \end{eqnarray*}}Let  $\alpha':=2e+k_e- 2b_e-2$. By Serre Duality theorem on $\Pp^1$, from above we have
$$h^1(A_e) = h^0(\Oc_{\Pp^1}(\alpha'))+h^0(\Oc_{\Pp^1}(\alpha' + e))+h^0(\Oc_{\Pp^1}(\alpha'+2e)).$$

\smallskip

\noindent
$\bullet$  If $ \alpha' +2e<0$, that is $k_e<2b_e+2-4e$, then $h^{1}(A_e)=0$ (observe that condition 
$k_e < 2b_e + 2 - 4e$ is compatible with $k_e > b_e - e$,  because of Assumptions \ref{ass:AB}-(ii)).

\smallskip

\noindent
$\bullet$  If $ \alpha' +e <0 \le \alpha' +2e$, i.e. $2b_e+2-4e \leq k_e < 2b_e+2-3e$, then
$$h^{1}(A_e)=h^0( \Oc_{\Pp^1}(\alpha'+2e)))= h^0(\Oc_{\Pp^1}(4e+k_e-2b_e-2))= 4e+k_e-2b_e-1.$$

\smallskip

\noindent
$\bullet$  If $ \alpha'<0 \le \alpha' + e$, equivalently $2b_e+2-3e\leq k_e < 2b_e+2-2e$, then
$$h^{1}(A_e)=  h^0(\Oc_{\Pp^1}(\alpha'+2e))+ h^0(\Oc_{\Pp^1}(\alpha'+e))=2\alpha'+3e+2=7e+2k_e-4b_e-2.$$

\smallskip

\noindent
$\bullet$  Finally, if $ \alpha' \ge 0$, which is $ k_e \ge 2b_e+2-2e$ then
$$h^{1}(A_e)=3\alpha'+3e+3=9e+3k_e-6b_e-3$$(notice that condition $k_e \ge 2b_e+2-2e$ is compatible with what computed in Remark \ref{rem:equivalence}; in other words one has 
$2b_e+2-2e < 4b_e - 6 e - 2 \le 4b_e - 6 e - 2 + h^1(\Ee)$ because of Assumptions \ref{ass:AB}-(ii)).
Hence $h^{1}(A_e)$ is as in \eqref{eq:h1A}.
\end{proof}

\begin{cor}\label{cor:van} Assumptions \ref{ass:AB}  and  $k_e< 2b_e+2-4e$ imply that 
any $\mathcal E_e \in {\rm Ext}^1(B_e,A_e)$ is such that $h^1(\mathcal E_e)=0$. 
\end{cor}

\begin{rem}\label{rem:added1} {\normalfont (1) Computations as in Remark \ref{rem:equivalence} show that $k_e < 2 b_e + 2 - 4 e$ 
implies $h^0(\mathcal E_e) = 4 b_e - k_e - 6 e + 5 \geq 2b_e - 2 e + 3$ which, from 
Assumption \ref{ass:AB}(iii) and $e \geq 2$, turns out to be greater than or equal to  $4e + 5 \geq 13$. Therefore, 
conditions $b_e \geq 3e +1$ and $b_e - e < k_e <  2 b_e + 2 - 4 e$ are sufficient for Assumptions \ref{ass:AB} to hold.

\noindent
(2) When moreover $b_e > 4e-4$, then $\frac{3b_e-4e}{2} < 2 b_e + 2 - 4e$ holds.  In this case, as observed in Remark \ref{rem:endom}-(2), 
Lemmas \ref{lem:ext1} and \ref{lem:autE} ensure that a general $\mathcal E_e \in {\rm Ext^1}(B_e, A_e)$  is indecompo-sable.
}
\end{rem}

\begin{rem}\label{rem:added2} {\normalfont As costumary, $0 \in {\rm Ext^1}(B_e, A_e)$ corresponds to the trivial bundle $A_e \oplus B_e$. When 
$k_e \geq 2b_e + 2 - 4e$ (i.e. when $h^1(A_e) >0$), a given $\mathcal E_e \in {\rm Ext^1}(B_e, A_e) \setminus \{0\}$ is non-special if and only if the coboundary 
map $\partial: H^0(B_e) \to H^1(A_e)$ (corresponding to the choice of $\mathcal E_e$) is surjective.  From \eqref{eq:comseq}, 
${\rm Im} (\partial) \cong {\rm Coker} \left\{H^0(\mathcal E_e) \stackrel{\rho}{\to} H^0(B_e)\right\}$; thus the surjectivity of $\partial$ 
can be geometrically interpreted with the fact that the linear system induced by the tautological line bundle $\Oc_{\Pp(\mathcal E_e)}(1)$ onto the 
section $\Sigma_e \subset \Pp(\mathcal E_e)$, corresponding to the quotient line bundle  
$\mathcal E_e \to \!\!\!\!\to B_e$, is not complete with ${\rm codim}_{H^0(\Oc_{\Sigma_e}(1))} ({\rm Im}(\rho)) = h^1(A_e)$. 
When $k_e \geq 2b_e + 2 - 4e$, it  is a very tricky problem to find conditions 
granting the existence of a sublocus $\mathcal U \subset {\rm Ext^1}(B_e, A_e)$ s.t. $h^1(\mathcal E_e)= 0$ for 
any $\mathcal E_e \in \mathcal U$.
}
\end{rem}

%

\section{$3$-dimensional scrolls over $\FF_e$ and their Hilbert schemes} \label{S:scrollsFe}
In this section, results from \S\,\ref{S:vbFe} are used for the study of suitable $3$-dimensional
scrolls over $\FF_e$ in projective spaces and of some components of their Hilbert schemes.

The choice of  $c_1({\mathcal E}_e)=3C_e+b_ef$
and of the integers $b_e, k_e$  (cf.\,Assumptions \ref{ass:AB}, \ref{ass:AB2}), give the first case for which the bundle $\mathcal E_e$ is both 
{\em uniform} and {\em very-ample}. 
Indeed, if $\mathcal E_e$ is assumed to be ample with $c_1({\mathcal E}_e)=3C_e+b_ef$ then 
the restriction of  ${{\mathcal E}_e}_{|f}$ to any $\pi_e$-fiber $f$ has to be ample; hence  
$${{\mathcal E}_e}_{|f}=\Oc_f(a)\oplus \Oc_f(b), \; \mbox{with} \; a, b >0$$and $a+b=3$ because  $c_1({\mathcal E}_e)f = 3$. Therefore, 
up to reordering, the only possibility is $a=2, b=1$ for any $\pi_e$-fiber $f$, i.e.  ${\mathcal E}_e$ is {\em uniform} (cf.\,e.g.\,\cite{OSS} and 
\cite[Def.\,3]{ApBr}). Moreover, $c_1({\mathcal E}_e) = 3 C_e + b_e f$, together with very-ampleness hypothesis, naturally lead to
Assumptions \ref{ass:AB}.

Indeed, one has the following necessary condition for very-ampleness:

\begin{prop}\label{prop:AB} (see\,\cite[Prop.\,7.2]{al-be}) Let ${\mathcal E}_e$ be a very-ample, rank-two vector bundle over $\FF_e$ such that
$$c_1({\mathcal E}_e) \num 3 C_e + b_e f \; \; {\rm and} \;\; c_2({\mathcal E}_e) = k_e.$$
Then ${\mathcal E}_e$ satisfies all the hypotheses in Assumptions \ref{ass:AB}. 
\end{prop}

\begin{rem}\label{rem:added3} {\normalfont (1) By Lemma \ref{lem:ext1}, when $k_e$ is such that $b_e-e < k_e < \frac{3b_e+2-5e}{2}$ the only bundle in ${\rm Ext^1}(B_e, A_e)$ is $\mathcal E_e := A_e \oplus B_e$. The  very-ampleness of $B_e$ and $A_e$ implies that of $\mathcal E_e := A_e \oplus B_e$, \cite[Lemma 3.2.3]{BESO}. On the other hand the  very-ampleness of $\mathcal E_e := A_e \oplus B_e$ implies the ampleness of $B_e$ and $A_e$, but on $\FF_{e}$  ampleness of a line bundle is equivalent to very-ampleness, 
 \cite[V,\,Cor.\,2.18]{H}, and thus $\mathcal E_e := A_e \oplus B_e$ very-ample implies that both $B_e$ and $A_e$ are very-ample.  Assumption \ref{ass:AB}(iii) (resp.,  $k_e < 2b_e - 4e$) is a necessary and sufficient condition for $B_e$ (resp., for $A_e$) to be very-ample. Since very-ampleness is an open condition, when $\dim({\rm Ext}^1(B_e, A_e)) >0$ and $k_e < 2b_e - 4e$ holds, then the general bundle $\mathcal E_e$ in ${\rm Ext}^1(B_e, A_e)$ is very-ample too. 

\smallskip

\noindent
(2) From the previous sections, condition $b_e - e < k_e < 2b_e - 4e$ is compatible because of Assumption \ref{ass:AB}(ii) and gives also that any $\mathcal E_e \in {\rm Ext}^1(B_e,A_e)$ is non-special.

\smallskip

\noindent
(3) Comparing Lemmas \ref{lem:ext1} and  \ref{lem:autE} with this new bound on $k_e$, we notice that $\frac{3b_e+2-5e}{2} < 2b_e - 4e$ holds if and only if $b_e \geq 3e +3$; similarly $\frac{3b_e+2-4e}{2} < 2b_e - 4e$ holds if and only if $b_e \geq 4e +3$ and, finally,  
$\frac{3b_e-4e}{2} < 2b_e - 4e$ holds if and only if $b \geq 4e + 1$. In particular, when $b_e \geq 4e+1$ and $\frac{3b_e-4e}{2} < k_e < 2b_e - 4e$, Lemma \ref{lem:autE} also ensures the existence of indecomposable bundles in ${\rm Ext}^1(B_e,A_e)$ (cf.\,Remark \ref{rem:added1}(2)). 
}
\end{rem}

From Remark \eqref{rem:added3}, it is clear that from now on we will focus on $b_e -e < k_e < 2b_e - 4e.$ In other words, Assumptions \ref{ass:AB} will be  replaced by:

%
%

\begin{ass}\label{ass:AB2} Let $e \ge 2$, $k_e$, $b_e$ be integers. Let 
${\mathcal E}_e$ be a rank-two vector bundle over
$\FF_e$ such that $$c_1({\mathcal E}_e) \num 3 C_e + b_e f, \;\; c_2({\mathcal E}_e) = k_e,$$ with 
\begin{equation}\label{eq:newbounds}
b_e \ge 3 e + 1 \;\; {\rm and} \;\; b_e-e < k_e < 2 b_e  - 4e.
\end{equation}
\end{ass}

Let $$\scrollcal{E_e}$$ be the 3-dimensional scroll over $\FF_e,$ 
and let $\pi_e: \FF_e \to \Pin{1}$ and $\varphi: \Pp({\mathcal E}_e) \to  \FF_e$ be the usual projections.
%
%
\begin{prop}\label{prop:X} Let  ${\mathcal E}_e $ be as in Assumptions \ref{ass:AB2}. Moreover, 
when $ \dim({\rm Ext}^1(B_e, A_e)) >0$, we further assume that ${\mathcal E}_e \in {\rm Ext}^1(B_e,A_e)$ is general. 
Then $\Oc_{\Pp(\Ee)}(1)$ defines an embedding
\begin{equation}\label{eq:Xe}
\Phi_e:= \Phi_{|\Oc_{\Pp(\Ee)}(1)|}: \, \Pp({\mathcal E}_e) \hookrightarrow  X_e \subset \Pin{n_e},
\end{equation} where $X_e = \Phi_e( \Pp(\Ee))$ is smooth, non-degenerate, of degree $d_e$, with
\begin{equation}\label{eq:nde}
n_e = 4b_e-k_e-6e+4 \geq 4e + 4 \geq 12 \;\;\; {\rm and} \;\;\; d_e = 6b_e-9e-k_e.
\end{equation}
Denoting by $(X_e, L_e) := (X_e, \Oc_{X_e} (H)) \cong \scrollcal{E_e}$, one also has 
\begin{equation}\label{eq:van}
h^i(X_e, L_e) = 0, \;\; i \geq 1.
\end{equation}
\end{prop}
\begin{proof} The very-ampleness of $L_e$ is equivalent to that of ${\mathcal E}_e$, and the latter follows from Remark \ref{rem:added3}(1) and Assumptions \ref{ass:AB2}. The formula on the degree $d_e$ of $X_e$ in \eqref{eq:nde} follows from \eqref{eq:d}.  From Leray's isomorphisms, Lemma \ref{lemma:comEAB} and Corollary \ref{cor:van} we get  \eqref{eq:van}. Finally, since $n_e+1: = h^0(X_e,L_e) = h^0(\FF_e, {\mathcal E}_e)$, then $n_e+1 \geq 4e+ 5 \geq 13$ follows from Remark \ref{rem:added1}(2) and the fact that $e \geq 2$. 
\end{proof}

%

%

\subsection{The component $\mathcal X_e$ of the Hilbert scheme containing [$X_e$]}\label{S:hilbX}

In what follows, we will be interested in studying the Hilbert scheme parametrizing subvarieties of 
$\Pp^{n_e}$ having the same Hilbert polynomial $P(T):=P_{X_e}(T) \in \mathbb{Q}[T]$ of $X_e$, which is the numerical polynomial defined by 
\begin{equation}\label{eq:numpol}
P(m)= \chi(X_e, mL_e)= \frac{1}{6}m^3L_e^3-\frac{1}{4}m^2L_e^2\cdot K+\frac{1}{12}mL_e\cdot(K^2+c_2)+\chi(\Oc_{X_e}),  \mbox{for all $m \in \mathbb{Z}$},
\end{equation}
as it follows from \cite[Example 15.2.5, pg 291]{Fu}. 
 
For basic terminology and facts on Hilbert schemes we follow, for instance, \cite{groth,Se,SO1}.

The scroll $X_e \subset \Pp^{n_e}$ corresponds to a point $[X_e] \in \mathcal H_3^{d_e,n_e}$, where $\mathcal H_3^{d_e,n_e}$ denotes the  Hilbert scheme parametrizing $3$-dimensional subvarieties of $\Pp^{n_e}$ with Hilbert polynomial $P(T)$ as above (in particular of degree $d_e$), where $n_e$ and $d_e$ are as in \eqref{eq:nde}. When $[X_e] \in \mathcal H_3^{d_e,n_e}$ is a smooth point, 
$X_e$ is said to be {\em unobstructed} in $\Pp^{n_e}$. Let 
\begin{equation}\label{eq:Ne}
N_e : = N_{X_e/\Pp^{n_e}}
\end{equation}be the normal bundle of $X_e$ in ${\mathbb P}^{n_e}$. 
From standard facts on Hilbert schemes (cf. e.g. \cite[Corollary 3.2.7]{Se}), one has  
\begin{equation}\label{eq:tanghilb}
T_{[X_e]} (\mathcal H_3^{d_e,n_e}) \cong H^0(N_e) 
\end{equation}and
\begin{equation}\label{eq:dimhilb}
h^0(N_e) - h^1(N_e) \le \dim_{[X_e]}(\mathcal H_3^{d_e,n_e}) \le h^0(N_e),  
\end{equation} where the left-most integer in \eqref{eq:dimhilb} is the {\em expected dimension} of $\mathcal H_3^{d_e,n_e}$ at $[X_e]$ and 
where equality holds on the right in  \eqref{eq:dimhilb} iff $X_e$ is unobstructed in $\Pp^{n_e}$.

The next result shows that $X_e$ is unobstructed and such that $[X_e]$ sits in an irreducible component of $\mathcal H_3^{d_e,n_e}$ with ``nice" behaviour.

\begin{theo}\label{thm:HilbertFe} 
There exists an irreducible component $\mathcal X_e \subseteq \mathcal H_3^{d_e,n_e}$, which is generically smooth and of (the expected) dimension
\begin{equation}\label{eq:expdim}
\dim(\mathcal X_e) =  n_e(n_e+1) + 3k_e - 2b_e +3e- 5,
\end{equation}
such that $[X_e]$ belongs to the smooth locus of $\mathcal X_e$.
\end{theo}
\begin{proof} By \eqref{eq:tanghilb} and \eqref{eq:dimhilb}, the statement will follow by showing that $H^i(X_e,N_e)=0$, for 
$i \geq 1$, and conducting an explicit computation of $h^0(X_e,N_e) = \chi(X_e, N_e).$

To do this, let
\begin{eqnarray}\label{eulersequscrollsuFe}
0\lra {\mathcal O}_{X_e} \lra {\mathcal O}_{X_e}(1)^{\oplus (n_e+1)} \lra T_{{{\mathbb P}^{n_e}}|{X}_e} \lra  0
\end{eqnarray}be the Euler sequence on ${\mathbb P}^{n_e}$ restricted to $X_e$. Since $(X_e,L_e)$ is a scroll over $\FF_e$,
\begin{eqnarray}
\label{quadratino}
H^{i}(X_e,{\mathcal O}_{X_e})= H^{i}(\FF_e,{\mathcal O}_{\FF_e})= 0,\quad \text{ for}\quad i\ge 1.
\end{eqnarray}

From \eqref{eq:van}, \brref{quadratino},  the cohomology sequence associated to
\brref{eulersequscrollsuFe} and from the fact that $X_e$ is non--degenerate, one has:  

\begin{equation}\label{eq:comtang}
h^0(X_e,T_{{ {\mathbb P}^{n_e}}|{X_e}}) = (n_e+1)^2 -1 \; \; \mbox{and} \;\;  
h^i(X_e,T_{{ {\mathbb P}^{n_e}}|{X_e}})=0, \; \mbox{for} \; i\ge 1. 
\end{equation}

The normal sequence
\begin{eqnarray}
\label{tangentsequ} 0\lra T_{X_e} \lra T_{{ {\mathbb P}^{n_e}}|{X_e}}
\lra N_e \lra 0
\end{eqnarray}
gives therefore 
\begin{eqnarray}
\label{cohomnormal}
H^{i}(X_e,N_e) \cong H^{i+1}(X_e,T_{X_e}) \qquad {\text {for} \quad  i\ge 1.}
\end{eqnarray}

\begin{claim}\label{cl:clar} $H^i(X_e, N_e) = 0$, for $i \geq 1$.
\end{claim}

\begin{proof}[Proof of Claim \ref{cl:clar}] \, From  \eqref{eq:comtang},  \eqref{tangentsequ}  and dimension reasons, one has $h^j(X_e,N_e) = 0$, for $j\ge 3$. 
For the other cohomology spaces, we can use \eqref{cohomnormal}.

In order to compute $H^{j}(X_e,T_{X_e}),\; j = 2,3,$ we use the scroll map
$\varphi:{\mathbb P}({\mathcal E}_e)\lra \FF_e$ and we consider the relative cotangent bundle sequence:
\begin{eqnarray}\label{relativctgbdl}
0\to \varphi^{*}({\Omega}^1_{\FF_e})\to {\Omega}^1_{X_e}
\to {\Omega}^1_{X_e|{\FF_e}} \lra 0.
\end{eqnarray}
From \brref{relativctgbdl} and the Whitney sum, one obtains
$$ c_1({\Omega}^1_{X_e})= c_1(\varphi^{*}({\Omega}^1_{\FF_e}))+c_1({\Omega}^1_{X_e|{\FF_e}})$$thus
$${\Omega}^1_{X_e|{\FF_e}}=K_{X_e}+\varphi^{*}(-c_1({\Omega}^1_{\FF_e}))=K_{X_e}+\varphi^{*}(-K_{\FF_e}).$$
The adjunction theoretic characterization of the scroll gives
$$K_{X_e}= -2L_e+\varphi^{*}(K_{\FF_e}+c_1({\mathcal E}_e))=-2L_e+\varphi^{*}(K_{\FF_e}+3C_{e}+b_e f)$$
thus$${\Omega}^1_{X|{\FF_e}}=K_{X_e}+\varphi^{*}(-K_{\FF_e})=
-2L_e+\varphi^{*}(3C_{e}+b_e f)$$which, combined with the dual of \brref{relativctgbdl}, gives
\begin{eqnarray}
\label{relative tgbdl}
0 \to 2L_e-\varphi^{*}(3C_{e}+b_e f) \to T_{X_e}
\to \varphi^{*}(T_{\FF_e}) \to 0.
\end{eqnarray}In what follows, we compute the cohomology of the left and right-most bundles in \eqref{relative tgbdl}.

\medskip

\noindent
(i) First we concentrate on $\varphi^{*}(T_{\FF_e})$. By Leray's isomorphism, one has  
$$H^i(\varphi^*(T_{\FF_e})) \cong H^i(T_{\FF_e}), \;\; \mbox{for any} \;\; i \ge 0.$$
Consider therefore the relative cotangent bundle sequence of $\pi_e: \FF_e \to \Pin{1}$
\begin{eqnarray}
\label{relativctgbdlFe}
0\to \pi_e^{*}{\Omega}^1_{\Pin{1}} \to {\Omega}^1_{\FF_e}
\to {\Omega}^1_{\FF_e|\Pin{1}} \to 0.
\end{eqnarray}
Since $ {\Omega}^1_{\FF_e|\Pin{1}} =K_{\FF_e}+ \pi_e^{*} \oofp{1}{2} = -2C_{e}-e f$, dualizing \brref{relativctgbdlFe} we get
\begin{eqnarray}
\label{1relativctgbdlFe}
0\to 2C_{e}+ e f \to T_{\FF_e} \to\pi_e^{*}T_{\Pin{1}} \to 0.
\end{eqnarray}Since $\pi_e^{*}T_{\Pin{1}} \cong \pi_e^{*} \Oc_{\Pp^1}(2)$, by Leray's isomorphism 
$$h^0(\pi_e^{*}T_{\Pin{1}}) = 3, \; h^i(\pi_e^{*}T_{\Pin{1}}) = 0, \mbox{for} \; i \ge 1.$$As in the proof of 
Lemma \ref{lem:h1A}, Leray's isomorphism gives 
$$h^i(2C_e + e f) = h^i(\Pp^1, [\Oc_{\Pp^1} \oplus \Oc_{\Pp^1} (-e) \oplus \Oc_{\Pp^1} (-2e)] 
\otimes \Oc_{\Pp^1}(e)), \; \mbox{for any} \; i \ge 1.$$Thus, 
$$h^0(2C_e+ef) = e+2, \; h^1(2C_e+ef) = e-1, \; h^j(2C_e+ef) = 0, \; \mbox{for} \; j \ge 2.$$From 
\cite[Lemma 10]{ma}, one has $$h^0(\FF_e, T_{\FF_e}) = e+5.$$Therefore, putting all together in the cohomology sequence associated to \brref{1relativctgbdlFe}, we get  
\begin{eqnarray}\label{eq:40}
h^0(X_e, \varphi^*(T_{\FF_e}))= h^0(\FF_e, T_{\FF_e}) = e+5, \nonumber\\
h^1(X_e, \varphi^*(T_{\FF_e}))= h^1(\FF_e, T_{\FF_e}) = e-1,\\ 
h^j(X_e, \varphi^*(T_{\FF_e}))= h^j(\FF_e, T_{\FF_e}) = 0, \; \mbox{for} \; j \ge 2. \nonumber
\end{eqnarray}

\medskip

\noindent
(ii) We now devote our attention to  the cohomology of $2L_e-\varphi^{*}(3C_{e}+b_e f)$ in \eqref{relative tgbdl}. Noticing that  $R^{i}\varphi_{*}(2L_e)=0$ for $i \ge 1$ (see \cite[Ex. 8.4, p. 253]{H}), projection formula and Leray's isomorphism give 
\begin{equation}\label{eq:5.aiuto}
H^i(X_e, 2L_e-\varphi^{*}(3C_{e}+b_e f)) \cong H^i({\FF_e}, Sym^2{\mathcal E}_e\otimes (-3C_{e}-b_e f)), \; \forall \; i \ge 0.
\end{equation}
Therefore
\begin{equation}\label{h3sym2E}
h^j(X_e, 2L_e-\varphi^{*}(3C_{e}+b_ef)) = 0, \; j \ge 3, 
\end{equation} for dimension reasons.

We now want to show that $H^2({\FF_e}, Sym^2{\mathcal E}_e\otimes (-3C_{e}-b_e f))=0.$ To do this, 
recall that $\Ee$ fits in the exact sequence 
\eqref{eq:al-be}, with $A_e$ and $B_e$ as in \eqref{eq:al-be3}. By \cite[5.16.(c), p. 127]{H},  there is a finite filtration of 
$Sym^2({\mathcal E}_e)$, 
$$Sym^2({\mathcal E}_e)=F^{0}\supseteq F^{1}\supseteq F^{2} \supseteq F^{3}=0$$
with quotients 
$$ F^{p}/F^{p+1} \cong Sym^{p}(A_e)\otimes Sym^{2-p}(B_e),$$for each $0 \le p \le 2$.
Hence $$ F^{0}/F^{1} \cong Sym^{0}(A_e) \otimes Sym^{2}(B_e)= 2B_e$$
$$ F^{1}/F^{2} \cong Sym^{1}(A_e)\otimes Sym^{1}(B_e)=A_e+ B_e$$
$$ F^{2}/F^{3} \cong Sym^{2}(A_e)\otimes Sym^{0}(B_e)=2A_e, \text{ that is } F^{2}=2A_e,$$since $F^3=0$. 
Thus, we get the following exact sequences
\begin{equation}\label{filtr1}
0 \to F^{1} \to Sym^2({\mathcal E}_e) \to 2B_e\to 0
\end{equation}
\begin{equation}\label{filtr2}
0 \to  F^{2} \to F^{1}  \to A_e+B_e\to 0
\end{equation}
\begin{equation}\label{filtr3}
F^{2} = 2 A_e
\end{equation}
Twisting \brref{filtr1},  \brref{filtr2} with $-c_1({\mathcal E}_e)=-3C_{e}-b_e f=-A_e-B_e$ and using  \brref{filtr3} we get 
\begin{equation}
\label{filtr1tw}
0 \to F^{1} (-3C_{e}-b_e f) \to Sym^2({\mathcal E}_e) \otimes (-3C_{e}-b_e f) \to B_e-A_e\to 0
\end{equation}
\begin{equation}
\label{filtr2tw}
0 \to A_e-B_e\to F^{1} (-3C_{e}-b_e f)  \to {\mathcal O}_{F_{e}}\to 0
\end{equation}

First we focus on \eqref{filtr2tw}; from \eqref{eq:A-B} and from the same arguments used in Lemma \ref{lem:ext1}, one gets  
$$h^{i}(A_e-B_e) = h^{i}({\mathbb P}^1, \oofp{1}{3b_e-2k_e-4e} \oplus \oofp{1}{3b_e-2k_e-5e});$$so, for 
dimension reasons,  $h^i(A_e-B_e) = 0$, for any $i \ge 2$. Since  moreover 
$h^{i}({\mathcal O}_{\FF_{e}})=0$ for $i\ge1$, then \brref{filtr2tw} gives 
\begin{equation}\label{eq:5.aiutob}
h^{2}( F^{1}(-3C_{e}-b_e f))=0.
\end{equation}

Passing to \eqref{filtr1tw} observe that, from \eqref{eq:B-A} and from the fact that $K_{\FF_e} \num - 2 C_e - (e+2)f$, one gets    
$$h^{2}(B_e-A_e) = h^{0}(-C_{e}+ (3b_e-2k_e-5e-2)f)=0.$$Thus, from \eqref{eq:5.aiutob}, \brref{filtr1tw} and \eqref{eq:5.aiuto}, 
one has 
\begin{equation}\label{h2sym2E}
h^{2}(\FF_{e}, Sym^{2}{\mathcal E}_{e}\otimes (-3C_{e}-b_e f))=h^2(X_e, 2L_e - \varphi^*(3C_e + b_ef)) = 0.
\end{equation}

\medskip

Using \eqref{eq:40}, \eqref{h3sym2E} and \eqref{h2sym2E} in the cohomology 
sequence associated to \brref{relative tgbdl}, we get 
\begin{equation}\label{eq:cohomTe}
h^j(X_e, T_{X_e}) = 0, \; \mbox{for} \; j \ge 2. 
\end{equation}Isomorphism \brref{cohomnormal} concludes the proof of Claim \ref{cl:clar}.
\end{proof}

Using \eqref{eq:tanghilb} and \eqref{eq:dimhilb}, Claim \ref{cl:clar} implies that there exists an irreducible component $\mathcal X_e$ of 
$\mathcal H_3^{d_e,n_e}$ containing $[X_e]$ as a smooth point.

Since smoothness is an open condition, $\mathcal X_e$ is generically smooth. 
Moreover, always from \eqref{eq:dimhilb} and Claim \ref{cl:clar}, it follows that 
$\dim(\mathcal X_e) = h^0(X_e,N_e)= \chi(N_e)$ i.e. $\mathcal X_e$ has the expected dimension.

The Hirzebruch-Riemann-Roch theorem gives
\begin{eqnarray}\label{chiN}
\chi(N_e) &=& \frac{1}{6}(n_1^3-3n_1n_2+3n_3)+
\frac{1}{4}c_1(n_1^2-2n_2) \\
& & + \frac{1}{12}(c_1^2+c_2)n_1 +(n_e-3)\chi({\mathcal O}_{X_e})\nonumber
\end{eqnarray}
where $n_i:=c_i(N_e)$ and $ c_i:=c_i(X_e)$. 

If $K:= K_{X_e}$, the Chern classes of $N_e$  can be obtained from \brref{tangentsequ}:
\begin{eqnarray}\label{valueniscrollP}
 n_1&=&K+(n_e+1)L_e;  \nonumber \\
\;\;\;\;\;\;\;\; n_2&=&\frac{1}{2}n_e(n_e+1)L_e^2+(n_e+1)L_eK+K^2-c_2; \\
n_3&=&\frac{1}{6}(n_e-1)n_e(n_e+1)L_e^3+\frac{1}{2}n_e(n_e+1)KL_e^2+ (n_e+1)K^2L_e \nonumber\\
& &-(n_e+1)c_2L_e-2c_2K+K^3-c_3. \nonumber
\end{eqnarray} The numerical invariants of $X_e$ can be easily computed by:
\begin{xalignat}{2}
  KL_e^2 &= -2d_e+4b_e-6e-6;&  K^2L_e &= 4d_e-14b_e+21e+20;\notag \\
  c_2L_e &= 2b_e-3e+10;& K^3 &= -8d_e+36b_e-54e-48;\notag \\
  -Kc_2 &= 24;&  c_3 &= 8. \nonumber
\end{xalignat}

Plugging these in \brref{valueniscrollP} and then in \brref{chiN}, one gets
$$\chi(N_e)=(d_e+3e-2b_e+5)n_e-5-24e+16b_e-3d_e.$$From \eqref{eq:nde}, one has 
$d_e = 6b_e-9e-k_e$; in particular 
$$d_e+3e-2b_e+5 = 4b_e-6e-k_e+5 = n_e+1,$$as it follows from \eqref{eq:nde}. Thus
$$\chi(N_e) = (n_e+1) n_e - 5 - 3 (6b_e-9e-k_e) - 24e + 16b_e = n_e(n_e+1) + 3 k_e - 2 b_e + 3e -5,$$
as in \eqref{eq:expdim}. 
\end{proof}

\begin{rem}\label{rem:flam} \normalfont{The proof of Theorem \ref{thm:HilbertFe} gives 
\begin{equation}\label{eq:comNe}
h^0(N_e) = n_e(n_e+1) + 3 k_e - 2 b_e + 3e -5, \;\; h^i(N_e) = 0, \; i\geq 1. 
\end{equation}Using  \eqref{eq:comtang} and \eqref{eq:comNe} in  
the exact sequence \eqref{tangentsequ}, one gets
\begin{equation}\label{eq:5.6}
\chi(T_{X_e}) = h^0(T_{\Pp^{n_e}|_{X_e}}) - h^0(N_e) = 6b_e - 4 k_e + 9 - 9e. 
\end{equation}Moreover, from \eqref{tangentsequ} and \eqref{eq:comtang}, one has:  
\begin{equation}\label{eq:5.flam}
0 \to H^0(T_{X_e}) \to H^0(T_{\Pp^{n_e}|_{X_e}}) \stackrel{\alpha}{\to} H^0(N_e) \stackrel{\beta}{\to} H^1(T_{X_e}) \to 0, 
\end{equation}
}
\end{rem}

In the sequel (cf. the proof  of Theorem \ref{thm:parcount} below) we will make use of the following consequences of Theorem \ref{thm:HilbertFe}, interpreted via \eqref{eq:5.flam}.

\begin{cor}\label{flam} When $\dim({\rm Ext}^1(B_e, A_e))=0$, one has 
$$h^0(T_{X_e}) = 6b_e-4k_e-8e+8, \;\; h^1(T_{X_e}) = e-1, \;\; h^j(T_{X_e}) = 0, \; \mbox{for} \; j \ge 2.$$In particular, 
\begin{equation}\label{flam2}
\dim({\rm Coker} (\alpha)) = e-1,
\end{equation}where $\alpha$ is the map in \eqref{eq:5.flam}.

\end{cor}
\begin{proof} From Lemma \ref{lem:ext1} and Remark \ref{rem:added3}(3), notice that $\dim({\rm Ext}^1(B_e, A_e))=0$ occurs 
when, either $b_e = 3e+1, 3e+2$ and for any $b_e-e < k_e < 2b_e -4e$, or for 
$b_e \geq 3e+3$ and $b_e-e < k_e < \frac{3b_e+2-5e}{2} < 2b_e - 4e$. 

Now $h^j(T_{X_e}) = 0$, for $j \ge 2$, is \eqref{eq:cohomTe} which more generally holds for 
any $b_e$, $k_e$ as in \eqref{eq:newbounds}. We thus concentrate on $h^j(T_{X_e})$, for  $j=0, 1$. 
Since $h^1(A_e-B_e) = \dim ({\rm Ext}^1(B_e, A_e)) = 0$, from \eqref{filtr2tw} one has 
$$h^0(F^1(-3C_e - b_e f)) = h^0(A_e-B_e) +1 = 6b_e - 4 k_e - 9 e + 3, \;\;\;\; h^1(F^1(-3C_e - b_e f)) = 0.$$
Passing to \eqref{filtr1tw}, from \eqref{eq:B-A} and Leray's isomorphism, one has 
$h^i(B_e-A_e)=0$ for any $i \ge 0$.  Thus 
$$h^i(Sym^{2}{\mathcal E}_{e}\otimes (-3C_{e}-b_e f)) = h^i(F^1(-3C_e - b_e f)) , \; \mbox{for} \; 0 \le i \le 2,$$ and thus 
$$h^0(Sym^{2}{\mathcal E}_{e}\otimes (-3C_{e}-b_e f)) = 6b_e - 4 k_e - 9 e + 3, \;\; h^1(Sym^{2}{\mathcal E}_{e}\otimes (-3C_{e}-b_e f))=0.$$
The cohomology sequence associated to  \eqref{relative tgbdl} along with  \eqref{eq:5.aiuto} and   \eqref{eq:40} gives the first part of the statement.

Finally, for \eqref{flam2}, it suffices to notice that the map $\beta$ in \eqref{eq:5.flam} is surjective.
\end{proof}

%
%

\section{The general point of the component $\mathcal X_e$}\label{S:genpoint}

In this section a description of  the general point of  $\mathcal X_e$, determined in Theorem \ref{thm:HilbertFe}, is presented. 
The following preliminary result shows that in general scrolls arising from Proposition \ref{prop:X} do not fill up $\mathcal X_e$.


\begin{theo}\label{thm:parcount}  Let $\mathcal Y_e$ be the locus in $\mathcal X_e$ filled up by  threefold scrolls $X_e$ as in Proposition 
\ref{prop:X}. Then

\noindent
(i) if $b_e-e < k_e < \frac{3b_e+2-5e}{2}$, one has ${\rm codim}_{\mathcal X_e} ( \mathcal Y_e) = e-1$, 

\noindent
(ii) if $\frac{3b_e+2-5e}{2} \le k_e \le 2b_e - 4e$, one has ${\rm codim}_{\mathcal X_e} ( \mathcal Y_e) \le e-1$. 
\end{theo}

\begin{proof} In case (i), from Lemma \ref{lem:ext1},  $\dim({\rm Ext}^1(B_e, A_e))=0$. Therefore  $X_e \cong \Pp(A_e \oplus B_e)$ is uniquely determined, so $\dim(\mathcal Y_e) = \dim({\rm Im}(\alpha))$, where $\alpha$ is  the map in \eqref{eq:5.flam}. Thus
$${\rm codim}_{\mathcal X_e} ( \mathcal Y_e) = \dim({\rm Coker}(\alpha))=e-1$$ where the last equality comes from 
\eqref{flam2}. 

In case (ii) we have $\dim({\rm Ext}^1(B_e, A_e))>0$; consider the following quantities.

\begin{itemize}
\item[(a)] Denote by $\tau_e$ the number of parameters counting isomorphism classes of projective bundles $\Pp(\Ee)$ 
as in Proposition \ref{prop:X}. In other words, $\tau_e$ takes into account {\em weak isomorphism classes} of extensions, which are parametrized by 
$\Pp({\rm Ext}^1(B_e,A_e))$ (cf. \cite[p. 31]{Frie}), see Lemma  \ref{lem:ext1} for the calculation of ${\rm Ext}^1(B_e,A_e)$. In particular, $\tau_e = \dim({\rm Ext}^1(B_e, A_e)) -1$ and,  
from Lemma \ref{lem:ext1}, this number is as follows: 
 \begin{equation}\label{eq:dimExt1b}
\tau_e: = \left\{
\begin{array}{ccc}
5e+2k_e-3b_e-2 & & \frac{3b_e+2-5e}{2} \le k_e <  \frac{3b_e+2-4e}{2}\\
9e+4k_e-6b_e-3 & &  \frac{3b_e+2-4e}{2} \le k_e < 2b_e - 4e 
\end{array}
\right.
\end{equation} 
(more precisely, note that if $\frac{3b_e + 2 -5e}{2} \le k_e < 2b_e -4e \leq \frac{3b_e + 2 -4e}{2}$, that is, when $3e + 3 \leq b_e \leq 4e +2$, then \eqref{eq:dimExt1b} simply reads 
$\tau_e: = 5e+2k_e-3b_e-2$).


\medskip 

\item[(b)] $G_{X_e} \subset PGL(n_e+1, \mathbb C)$ denotes the {\em projective stabilizer} of $X_e \subset \Pp^{n_e}$, i.e. the subgroup of projectivities of $\Pp^{n_e}$ fixing $X_e$. In particular (cf. \eqref{tangentsequ}) 
\begin{equation}\label{eq:dimorbit}
\dim(PGL(n_e+1, \mathbb C)) - \dim (G_{X_e}) = n_e(n_e + 2) - h^0(T_{X_e})
\end{equation}is the dimension of the full orbit of $X_e \subset \Pp^{n_e}$ under the action of all the projective transformations of $\Pp^{n_e}$. 
This equals $\dim({\rm Im}(\alpha))$, where $\alpha$ is the map in \eqref{eq:5.flam}.
\end{itemize}

\medskip

The rest of the proof now reduces to a parameter computation to obtain a lower bound for the dimension of $\mathcal Y_e$. From the exact sequence \eqref{eq:al-be}, we observe that:

\vskip 3pt 

\noindent
(*) the line bundle $A_e$ is uniquely determined on $\FF_e$, since $A_e \cong \Oc_{\FF_e}(2C_e) \otimes {\pi_e}^*\Oc_{\Pp^1}(2b_e-k_e-2e)$;   

\vskip 3pt 

\noindent(**) the line bundle $B_e$ is uniquely determined on $\FF_e$, similarly.

Let us compute how many parameters are needed to describe $\mathcal Y_e$. To do this, we have to add up the following quantities:

\vskip 2pt 

\noindent
1) $0$ parameters for $A_e$ on $\FF_e$, by (*);

\vskip 3pt 

\noindent
2) $0$ parameters for  $B_e$, by (**);

\vskip 2pt 

\noindent
3) $\tau_e$ as in \eqref{eq:dimExt1b},  for  isomorphism classes of  $\Pp({\mathcal E}_e)$; 


\vskip 2pt 

\noindent
4)  $n_e(n_e + 2) - h^0(T_{X_e})$, as in \eqref{eq:dimorbit}, for the dimension of the full orbit of $X_e \subset \Pp^{n_e}$ chosen.

\vskip 2pt 

\noindent
Thus, 
\begin{equation}\label{eq:dimY}
\dim(\mathcal Y_e) = \tau_e + n_e(n_e+2) - \dim(G_{X_e})
\end{equation}The next step is to find an upper bound for $\dim(G_{X_e})$. It is clear that there is an obvious inclusion
\begin{equation}\label{eq:GX}
G_{X_e} \hookrightarrow Aut(X_e),
\end{equation} where $Aut(X_e)$ denotes the algebraic group of abstract automorphisms of $X_e$. 
Since $X_e$, as an abstract variety, is isomorphic to $\Pp({\mathcal E}_e)$ over $\FF_e$, then 
$$\dim(Aut(X_e)) = \dim (Aut(\FF_e)) + \dim(Aut_{\FF_e} (\Pp({\mathcal E}_e))),$$where $Aut_{\FF_e} (\Pp({\mathcal E}_e))$ denotes the group of automorphisms of $\Pp({\mathcal E}_e)$ fixing the base (cf. e.g. \cite{ma}). From the fact that $Aut(\FF_e)$ is an algebraic group, in particular smooth, it follows that$$\dim(Aut(\FF_e)) = h^0(\FF_e, T_{\FF_e}) = e+5$$since 
$e \ge 2$ (cf. \cite[Lemma 10, p.\,105]{ma}). On the other hand, $\dim(Aut_{\FF_e} (\Pp({\mathcal E}_e))) = h^0({\mathcal E}_e \otimes {\mathcal E}_e^{\vee}) -1$, since $Aut_{\FF_e} (\Pp({\mathcal E}_e))$ are given by endomorphisms of the projective bundle.

To sum up, $$\dim(Aut(X_e)) = h^0({\mathcal E}_e \otimes {\mathcal E}_e^{\vee}) + 4 +e.$$
From \eqref{eq:GX}, $\dim(G_{X_e}) \leq \dim(Aut(X_e))$, then from \eqref{eq:dimY} we deduce
\begin{equation}\label{eq:dimY2}
\dim(\mathcal Y_e) \geq  \tau_e + n_e(n_e+2) - h^0({\mathcal E}_e \otimes {\mathcal E}_e^{\vee}) - 4 -e.
\end{equation} 

According to Lemma \ref{lem:autE}, one has 
$$h^0({\mathcal E}_e \otimes {\mathcal E}_e^{\vee}) = \left\{
\begin{array}{ccl}
3b_e-2k_e -4e+2 & &  {\rm for} \;\; \frac{3b_e+2-5e}{2} \leq k_e < 2b_e  - 4e \leq \frac{3b_e-4e}{2}\\
& & \\
1 & &  {\rm for} \;\; \frac{3b_e-4e}{2} \le k_e < 2b_e-4e,
\end{array}
\right.$$
As for  $\tau_e$, we use  \eqref{eq:dimExt1b} and hence we get

\smallskip

\begin{itemize}
\item[(a)] for $\frac{3b_e+2-5e}{2} \le k_e < \frac{3b_e-4e}{2}$, $\tau_e  = 5e + 2 k_e - 3 b_e -2$ 
and $h^0(\mathcal E \otimes \mathcal E^{\vee}) = 3b_e-2k_e -4e+2$,  
\item[(b)] for  $\frac{3b_e-4e}{2} \le k_e < \frac{3b_e+2 -4e}{2}$, $\tau_e  = 5e+2k_e-3b_e-2 $ 
and $h^0(\mathcal E \otimes \mathcal E^{\vee}) = 1$; 
\item[(c)] for $\frac{3b_e+2-4e}{2} \le k_e < 2b_e-4e$, $\tau_e  = 9e+4k_e-6b_e-3$ 
and $h^0(\mathcal E \otimes \mathcal E^{\vee}) = 1$. 
\end{itemize}

\smallskip

In all cases, from \eqref{eq:dimY2} we get $\dim(\mathcal Y_e) \geq  n_e(n_e+2) - 6b_e + 4 k_e+8e-8$. From \eqref{eq:expdim}, we get
{\small
\begin{eqnarray*}
{\rm codim}_{\mathcal X_e} ( \mathcal Y_e)  & = & \dim(\mathcal X_e) - \dim(\mathcal Y_e) \\
 & \le &  n_e(n_e+1) + 3k_e - 2b_e +3e- 5 - (n_e(n_e+2) - 6b_e + 4 k_e+8e-8) = e-1.
\end{eqnarray*}
}
\end{proof}

%
%

\subsection{A candidate for the general point of $\mathcal X_e$}\label{S:5.3new} From Theorem \ref{thm:parcount}, we need to exhibit a smooth variety in $\Pp^{n_e}$ which is a candidate to represent the general point of 
$\mathcal X_e$ as in Theorem \ref{thm:HilbertFe}. In other words, this candidate must flatly degenerate in $\Pp^{n_e}$ to the threefold scroll $X_e$,  corresponding to $[X_e] \in \mathcal Y_e$ general,  in such a way that the base-scheme of this flat, embedded degeneration  is contained in 
$\mathcal X_e$.  

In this section we first construct this candidate and  analyze some of its properties similar to those investigated for $X_e$ in 
\S's\;\ref{S:vbFe}, \ref{S:scrollsFe}. In  \S\;\ref{S:genpointeps}, we show that this candidate 
actually corresponds to the general point of $\mathcal X_e$. 

For $e \ge 2$ integer, consider
\begin{equation}\label{eq:eps1}
\epsilon = 0, 1 \; \mbox{according to} \; \epsilon \equiv e \pmod{2}. 
\end{equation}
Consider the Hirzebruch surface $\FF_{\epsilon}$,  let 
$\pi_{\epsilon} : \FF_{\epsilon} \to \Pp^1$ be the natural projection and  
let $C_{\epsilon}$ be the unique section of $\FF_{\epsilon}$ corresponding to 
$\Oc_{\Pp^1} \oplus\Oc_{\Pp^1}(-\epsilon) \to\!\!\!\to \Oc_{\Pp^1}(-\epsilon)$ on $\Pp^1$. Thus  
$C_{\epsilon}^2 = - \epsilon$.

With notation as in Assumptions \ref{ass:AB2}, consider 
\begin{equation}\label{eq:eps34}
b_{\epsilon} := b_e - \frac{3(e-\epsilon)}{2}  \;\; \mbox{and} \;\; k_{\epsilon} := k_e.
\end{equation}
This choice of $b_{\epsilon} $ is needed in order to ensure that the Hilbert polynomial  data (in particular the degree) of $X_{\epsilon}$ are the same as those of  $X_e$, as it will become clear  in  \eqref{eq:gradoeps}.

\begin{lemma}\label{lem:AB1epsAB2eps} With \eqref{eq:eps34} above,  conditions \eqref{eq:newbounds} on 
$b_e$ and $k_e$ read as

\begin{equation}\label{eq:newboundseps}
b_{\epsilon} \ge \frac{3}{2} (e +\epsilon) + 1 \ge \frac{3\epsilon}{2} + 4 \; \;\; {\rm and} \;\;\; b_{\epsilon} - \epsilon < 
b_{\epsilon} + \frac{(e- 3 \epsilon)}{2}  
< k_{\epsilon} < 2b_{\epsilon} - 3 \epsilon -e.
\end{equation}

\end{lemma}

\begin{proof} The proof is given by straightforward computations using \eqref{eq:newbounds} and \eqref{eq:eps34}. 
Indeed, by \eqref{eq:eps34}, $b_e \ge 3e+1$ in  \eqref{eq:newbounds} reads as 
$b_{\epsilon}+ \frac{3(e-\epsilon)}{2} \ge 3e +1$ which is $b_{\epsilon} \ge \frac{3}{2} e + 1+\frac{ 3\epsilon}{2}$; 
the latter is greater than or equal to $\frac{ 3\epsilon}{2} + 4$ since $e \ge 2$ and from hypotheses on $\epsilon$. Similarly, one has 
$b_{\epsilon} + \frac{(e- 3 \epsilon)}{2} = b_{\epsilon} - \epsilon + \frac{(e- \epsilon)}{2} > b_{\epsilon} - \epsilon$ for the same reasons. 

Using  $b_{\epsilon} = b_e - \frac{3(e-\epsilon)}{2}$, one finds 
\begin{equation}\label{eq:lowerbound}
b_e - e = b_{\epsilon} + \frac{1}{2} (e- 3 \epsilon).
\end{equation}


Using \eqref{eq:lowerbound}, one gets 
\begin{equation}\label{eq:upperbound}
2b_e - 4 e  = 2(b_e-e) - 2e = 2 b_{\epsilon} - 3\epsilon - e.
\end{equation} 


Since from \eqref{eq:eps34} one has $k_{\epsilon} = k_e$, then one concludes by  \eqref{eq:newbounds}. 
\end{proof}

Consider now the following line bundles on $\FF_{\epsilon}$ (cf. \eqref{eq:al-be3}): 
\begin{equation}\label{eq:eps5}
A_{\epsilon} \num 2 C_{\epsilon} + (2b_{\epsilon} - k_{\epsilon} - 2 {\epsilon}) f 
\end{equation} 
and 
\begin{equation}\label{eq:eps6}
B_{\epsilon} \num C_{\epsilon} + (k_{\epsilon} - b_{\epsilon} + 2 {\epsilon}) f.
\end{equation}


\begin{rem}\label{rem:ABepsvamp} {\normalfont Notice that, with these choices, both $A_{\epsilon}$ and $B_{\epsilon}$ are very-ample. 
Indeed, from \cite[V\,Cor.\,2.18]{H}, $B_{\epsilon}$ is very-ample 
if and only if $k_{\epsilon} > b_{\epsilon} - \epsilon$, wheras $A_{\epsilon}$ is very-ample 
if and only if $k_{\epsilon} < 2 b_{\epsilon} - 4\epsilon$. Both conditions are implied by \eqref{eq:newboundseps}, since $e \geq 2$. 
} 
\end{rem}

 
As in \eqref{eq:al-be}, we consider $\mathcal E_{\epsilon}$ a rank--two vector bundle on $\FF_{\epsilon}$ fitting in the exact sequence 
\begin{equation}\label{eq:eps7}
0 \to A_{\epsilon} \to \mathcal E_{\epsilon} \to B_{\epsilon} \to 0.
\end{equation}
Thus 
$$c_1(\mathcal E_{\epsilon}) = A_{\epsilon} + B_{\epsilon} \num 3 C_{\epsilon} + b_{\epsilon} f \;\;\; \mbox{and} \;\;\; c_2(\mathcal E_{\epsilon}) = A_{\epsilon}B_{\epsilon} = k_{\epsilon} = k_e.$$
From \eqref{eq:d} one has $\deg(\mathcal E_{\epsilon}) = (3 C_{\epsilon} + b_{\epsilon} f)^2 - k_{\epsilon} = - 9 \epsilon + 6 b_{\epsilon} - k_{\epsilon}$. Thus \eqref{eq:eps34} gives 
\begin{equation}\label{eq:gradoeps}
\deg(\mathcal E_{\epsilon}) = 6b_e  - 9 e - k_e = d_e,  
\end{equation}where $d_e = \deg(\Ee)$ is as in \eqref{eq:nde}. 

\medskip

Now ${\rm Ext}^1(B_{\epsilon},A_{\epsilon}) \cong H^1(A_{\epsilon}-B_{\epsilon})$, 
where $A_{\epsilon}-B_{\epsilon} \num C_{\epsilon} + (3b_{\epsilon}-2k_{\epsilon}-4\epsilon)f$ 
from \eqref{eq:eps5}, \eqref{eq:eps6}. In particular, ${\pi_{\epsilon}}_*(A_{\epsilon}-B_{\epsilon}) \cong 
\left(\Oc_{\Pp^1}\oplus \Oc_{\Pp^1}(-\epsilon)\right)\otimes \Oc_{\Pp^1}(3b_{\epsilon}-2k_{\epsilon}-4\epsilon)$. We then use 
similar computations as in the proofs of Lemmas \ref{lem:ext1} and \ref{lem:autE},  in the range  \eqref{eq:newboundseps} for $k_{\epsilon}$ of interest for 
us (recall Lemma \ref{lem:AB1epsAB2eps}), and we get:

\begin{equation}\label{eq:dimExt1eps}
\dim({\rm Ext}^1(B_{\epsilon},A_{\epsilon})) = \left\{
\begin{array}{ccc}
0 & & {\rm for} \; b_{\epsilon} + \frac{e- 3 \epsilon}{2} < k_{\epsilon} <  \frac{3b_{\epsilon}+2-5\epsilon}{2}\\
& & \\
4k_{\epsilon} - 6 b_{\epsilon} - 2 + 9\epsilon & &    {\rm for} \; \frac{3b_{\epsilon}+2-5\epsilon}{2} \le k_{\epsilon} < 2b_{\epsilon}-3\epsilon-e
\end{array}
\right.
\end{equation} and 
 
{\footnotesize
\begin{equation}\label{eq:aut01}
h^0({\mathcal E}_{\epsilon} \otimes \mathcal E_{\epsilon}^{\vee}) = \left\{
\begin{array}{ccl}
6b_{\epsilon}-4k_{\epsilon} -9\epsilon + 4 & & {\rm for} \;\; b_{\epsilon} + \frac{e- 3 \epsilon}{2} < k_{\epsilon} < \frac{3b_{\epsilon}+2-5\epsilon}{2}\\
 & & \\
1 & &  {\rm for} \;\; \frac{3b_{\epsilon}+2-5\epsilon}{2} \le k_{\epsilon} <  2b_{\epsilon}-3\epsilon -e\;\; {\rm and} \;\; {\mathcal E}_{\epsilon} \; {\rm general}; 
\end{array}
\right.
\end{equation}}(the reader will easily realize that the distinction of cases in \eqref{eq:dimExt1eps} and in \eqref{eq:aut01} occurs when $\frac{3b_{\epsilon}+2-5\epsilon}{2} < 2b_{\epsilon}-3\epsilon-e$, 
that is for $b_{\epsilon} > 2 e + \epsilon + 2$, i.e. for $b_e > \frac{7e-\epsilon}{2} +2$ as it follows from \eqref{eq:eps34}; otherwise, only the first case in \eqref{eq:dimExt1eps} and in \eqref{eq:aut01} occurs, but we will not dwell on this).

\medskip  


Using \eqref{eq:eps7} and same reasoning as in Lemma \ref{lemma:comEAB}, under numerical assumptions 
\eqref{eq:newboundseps} we get 
\begin{equation}\label{eq:h1Beps}
h^j(B_{\epsilon}) = 0, \;\; \mbox{for} \;\; j \ge 1. 
\end{equation}
Using the same strategy as in Lemma \ref{lemma:comEAB}, 
considerations similar to \eqref{eq:comseq}, \eqref{eq:h1Eh1A} and \eqref{eq:h0E} can be done for 
$\mathcal E_{\epsilon}$ and one gets
\begin{equation}\label{eq:h1Eh1Aeps} 
h^1(\mathcal E_{\epsilon}) \le h^1(A_{\epsilon}) \;\; {\rm and} \;\; h^0(\mathcal E_{\epsilon})  = 4 b_{\epsilon} - k_{\epsilon} - 6 {\epsilon} 
+ 5 + h^1(\mathcal E_{\epsilon})
\end{equation}In particular, from \eqref{eq:eps34}, one has: 
\begin{equation}\label{eq:h0Eeps}
h^0(\mathcal E_{\epsilon}) = 4 b_e - k_e - 6 e + 5 + h^1(\mathcal E_{\epsilon}) = (n_e +1) + h^1(\mathcal E_{\epsilon}),
\end{equation}where $n_e = \chi(\Ee)-1 = h^0(\mathcal E_e)-1$ as in \eqref{eq:nde}.


To compute $h^1(A_{\epsilon})$ we follow the same strategy as in Lemma \ref{lem:h1A}. Since ${\pi_{\epsilon}}_*(A_{\epsilon}) \cong
Sym^2(\Oc_{\Pp^1}\oplus \Oc_{\Pp^1}(-\epsilon))\otimes \Oc_{\Pp^1}(2b_{\epsilon}-k_{\epsilon}-2\epsilon)$, by Leray's isomorphism 
one gets that $h^1(A_{\epsilon}) = h^1({\pi_{\epsilon}}_*(A_{\epsilon})) = 0$ as soon as $k_{\epsilon} < 2b_{\epsilon}+2-4\epsilon$. 
Considering the upper--bound for $k_{\epsilon}$ in \eqref{eq:newboundseps}, we notice that 
$2b_{\epsilon} - 3\epsilon - e < 2b_{\epsilon}+2-4\epsilon$; in other words, for $k_{\epsilon}$  as
in  \eqref{eq:newboundseps}, one has
\begin{equation}\label{eq:h1Aeps}
h^1(A_{\epsilon}) = 0. 
\end{equation}

As in Corollaries \ref{cor:dimExt1}, \ref{cor:van}, we get therefore:

\begin{cor}\label{cor:spliteps} Assumptions \eqref{eq:newboundseps} imply that any ${\mathcal E}_{\epsilon} \in {\rm Ext}^1(B_{\epsilon},A_{\epsilon})$ is such that $h^1({\mathcal E}_{\epsilon}) = 0$. In particular, 
\begin{equation}\label{eq:neps}
h^0(\mathcal E_{\epsilon}) = n_e +1,
\end{equation}with $n_e$ as in \eqref{eq:nde}. 
\end{cor}

\begin{proof} \eqref{eq:neps} follows from \eqref{eq:h0Eeps} and from what proved above. 
\end{proof}

Let now $(\Pp(\mathcal E_{\epsilon}), \Oc_{\Pp(\mathcal E_{\epsilon})}(1))$ be the 3-dimensional scroll over $\FF_{\epsilon}$ associated to any $\mathcal E_{\epsilon}$ as above. From Remark \ref{rem:ABepsvamp}, $A_{\epsilon} \oplus B_{\epsilon}$ is very-ample. Since very-ampleness is 
an open condition, when $\dim({\rm Ext}^1(B_{\epsilon}, A_{\epsilon})) >0$, the general $\mathcal E_{\epsilon} \in {\rm Ext}^1(B_{\epsilon}, A_{\epsilon})$ is also very-ample and thus $\Oc_{\Pp(\mathcal E_{\epsilon})}(1)$ defines an embedding
\begin{equation}\label{eq:Xeps}
\Phi_{\epsilon} := \Phi_{|\Oc_{\Pp(\mathcal E_{\epsilon})}(1)|}: \, \Pp({\mathcal E}_{\epsilon}) \hookrightarrow  X_{\epsilon} \subset \Pin{n_e},
\end{equation} 
(see  \eqref{eq:neps}), where $X_{\epsilon}:= \Phi_{\epsilon} (\Pp({\mathcal E}_{\epsilon})) $ is smooth, non-degenerate of degree $d_e$  (cf. \eqref{eq:gradoeps}). Moreover, letting $(X_{\epsilon}, L_{\epsilon}) := (X_{\epsilon}, \Oc_{X_{\epsilon}} (H)) \cong (\Pp(\mathcal E_{\epsilon}), \Oc_{\Pp(\mathcal E_{\epsilon})}(1))$, one has $h^i(X_{\epsilon}, L_{\epsilon}) = 0,\;i \geq 1$.

One can easily see that  $X_{\epsilon}$ and $X_e$ have the same Hilbert polynomial $P(T)$, defined by \eqref{eq:numpol}, so $X_{\epsilon} \subset \Pp^{n_e}$ corresponds to a point $[X_{\epsilon}]$ of the Hilbert scheme  $\mathcal H_3^{d_e,n_e}$ as in \S\,\ref{S:hilbX}.

\begin{prop}\label{prop:5.5eps} For any $\epsilon$, $b_{\epsilon}$ and $k_{\epsilon}$ as in \eqref{eq:eps1}, \eqref{eq:eps34} and 
\eqref{eq:newboundseps}, there exists an irreducible component $\mathcal X_{\epsilon} \subseteq \mathcal H_3^{d_e,n_e}$ 
which is generically smooth, of (the expected) dimension
\begin{equation}\label{eq:expdimeps}
\dim(\mathcal X_{\epsilon}) =  n_e(n_e+1) + 3k_{\epsilon} - 2b_{\epsilon} +3 \epsilon- 5,
\end{equation}such that $[X_{\epsilon}]$ belongs to the smooth locus of $\mathcal X_{\epsilon}$.
Moreover, the general point of $\mathcal X_{\epsilon}$ parametrizes a scroll $X_{\epsilon}$ 
as in  \eqref{eq:Xeps}.
\end{prop}

\begin{rem}\label{rem:dimeps} \normalfont{
Notice that, from \eqref{eq:eps34}, the right hand side of the equality in \eqref{eq:expdimeps} coincides with 
that of  \eqref{eq:expdim}, in other words $\dim(\mathcal X_{\epsilon}) = \dim (\mathcal X_e)$. 
}
\end{rem}

\begin{proof} [Proof of Proposition \ref{prop:5.5eps}] Let $N_{\epsilon} : = N_{X_{\epsilon}/\Pp^{n_e}}$ denote the normal bundle of $X_{\epsilon}$ in ${\mathbb P}^{n_e}$. As in Theorems \ref{thm:HilbertFe}, \ref{thm:parcount}, the statement will follow by proving 
the following intermediate steps: 

\noindent
$(a)$ show that $H^i(X_{\epsilon},N_{\epsilon})= (0)$, for $i \geq 1$, 

\noindent 
$(b)$ conduct an explicit computation of $h^0(X_{\epsilon},N_{\epsilon}) = \chi(X_{\epsilon},N_{\epsilon})$, 

\noindent
$(c)$ perform a parameter computation to estimate the dimension of the locus $\mathcal Y_{\epsilon}$ filled up 
by scrolls $X_{\epsilon}$ as in \eqref{eq:Xeps}. Therefore $\dim(\mathcal Y_{\epsilon})$ 
gives  a lower bound for $\dim(\mathcal X_{\epsilon})$. Finally,

\noindent
$(d)$ show that $\dim(\mathcal Y_{\epsilon})$ equals the number in \eqref{eq:expdimeps}.

\bigskip

\noindent
\underline{Case $\epsilon=1$}. From \eqref{eq:newboundseps}, we have $5 \le b_1 \le b_1 + \frac{e-3}{2} < k_1 < 2b_1 - 3 -e $, indeed 
by  \eqref{eq:eps1} the case $e$ odd gives $e \geq 3$. Notice that the upper and lower bound are compatible since, by \eqref{eq:newboundseps}, 
$b_1 \geq \frac{3}{2}(e+1)+1$. Using \eqref{eq:eps5}, \eqref{eq:eps6}, we get 
$$A_1 \num 2 C_1 + (2 b_1 - k_1 - 2)f  \;\;\; {\rm and} \;\;\; B_1 \num C_1 + (k_1 - b_1 + 2)f.$$All 
steps $(a)$-$(d)$ are already proved in \cite[Prop.\,5.5,\,Thm.\,5.7]{be-fa-fl} (cases considered here all come from 
cases therein coming from the first line of \cite[(16)\,in\,Lemma\,3.7]{be-fa-fl}).

\bigskip 

\noindent
\underline{Case $\epsilon=0$}. In this case, we have $ b_0 + \frac{e}{2} < k_0 < 2 b_0 -e$ where, from 
\eqref{eq:newbounds}, $b_0 >3$ for $e \ge 2$ even and the upper and  lower bound on $k_0$ are compatible. By \eqref{eq:eps5}, \eqref{eq:eps6}, 
we have$$A_0 \num 2 C_0 + (2 b_0 - k_0 - 2) f \;\;\; {\rm and} \;\;\; B_0 \num C_0 + (k_0 - b_0 + 2)f,$$where $C_0$ and $f$ 
are generators of the two different rulings on $\FF_0$.  

For Steps $(a)$ and $(b)$, we will use the same strategy of Theorem \ref{thm:HilbertFe}. 
By Corollary \ref{cor:spliteps}, $H^i(X_0,L_0)= 0$, for $i \geq 1$.

Thus, using the Euler sequence restricted to $X_0$ as in \eqref{eulersequscrollsuFe}, 
the fact that $(X_0, L_0)$ is a scroll over $\FF_0$,  non--degenerate in $\Pp^{n_e}$ 
(cf. \eqref{quadratino} and \eqref{eq:comtang}) and the normal sequence of 
$X_0 \subset \Pp^{n_e}$ as in \eqref{tangentsequ}, we get 

\begin{equation}\label{cohomnormal0}
H^{i}(X_0,N_0) \cong H^{i+1}(X_0,T_{X_0}) \qquad {\text {for} \quad  i\ge 1.}
\end{equation}Consequently $h^3(X_0,N_0) = 0$ for dimension reasons; for $h^1(X_0,N_0), \; h^2(X_0,N_0)$, we can use \eqref{cohomnormal0}. 

In order to compute $H^{j}(X_0,T_{X_0}),\; j = 2,3,$  let  $\varphi:{\mathbb P}({\mathcal E}_0)\lra \FF_0$ be the scroll map. We use the relative cotangent bundle sequence as in \eqref{relativctgbdl} and adjunction on $X_0$ to get, as in \eqref{relative tgbdl},  the exact sequence  
\begin{equation}\label{relative tgbdl0}
0 \to 2L_0-\varphi^{*}(3C_0+b_0 f) \to T_{X_0} \to \varphi^{*}(T_{\FF_0}) \to 0.
\end{equation}

By Leray's isomorphism, one has $H^j(\varphi^*(T_{\FF_0})) \cong H^j(T_{\FF_0}), \;\; \mbox{for any} \;\; j \ge 0$. 
Since $\FF_0 \cong \Pp^1 \times \Pp^1$, then $h^j(T_{\FF_0}) = 2 h^j (\Oc_{\Pp^1}(2))$, for any $j \ge 0$. Thus, 
\begin{equation}\label{eq:400}
h^0(X_0, \varphi^*(T_{\FF_0}))= h^0(\FF_0, T_{\FF_0}) = 6\;\;{\rm and}\;\; 
h^j(X_0, \varphi^*(T_{\FF_0}))= h^j(\FF_0, T_{\FF_0}) = 0, \; \mbox{for} \; j \ge 1.
\end{equation} 

For the cohomology of $2L_0-\varphi^{*}(3C_0+b_0 f)$, since $R^{i}\varphi_{*}(2L_0)=0$ for $i \ge 1$
(see \cite[Ex. 8.4, p. 253]{H}), projection formula and Leray's isomorphism give 
\begin{equation}\label{eq:5.aiuto0}
H^i(X_0, 2L_0-\varphi^{*}(3C_0+b_0 f)) \cong H^i({\FF_0}, Sym^2{\mathcal E}_0\otimes (-3C_0-b_0 f)), \; \forall \; i \ge 0.
\end{equation}Therefore
\begin{equation}\label{h3sym2E0}
h^j(X_0, 2L_0-\varphi^{*}(3C_{0}+b_0f)) = 0, \; j \ge 3, 
\end{equation} for dimension reasons. Finally, we use filtrations as in \eqref{filtr1}, \eqref{filtr2},  \eqref{filtr3} and 
argue as in the proof of Claim \ref{cl:clar}-(ii), to get also 
\begin{equation}\label{h2sym2E0}
h^{2}(\FF_0, Sym^{2}{\mathcal E}_0\otimes (-3C_0-b_0 f))=h^2(X_0, 2L_0 - \varphi^*(3C_0 + b_0f)) = 0.
\end{equation} From \eqref{relative tgbdl0}, using \eqref{eq:400}, \eqref{eq:5.aiuto0} and \eqref{h3sym2E0}, we deduce 
that $h^j (X_0,T_{X_0}) = 0$, for any $j \ge 2$, so from \eqref{cohomnormal0} we get $h^i(N_0) = 0$, for $i \ge 1$. 

In particular, generic smoothness of $\mathcal X_0$ and the fact that it has the expected dimension 
follow from \eqref{eq:tanghilb}, \eqref{eq:dimhilb}. 

To compute the expected dimension (i.e. Step $(b)$), we use the Hirzebruch-Riemann-Roch theorem as in \eqref{chiN}, with 
values as in \eqref{valueniscrollP}. This gives 
$$h^0(N_0) = \chi(N_0) = (d_0 - 2 b_0 + 5) n_0 - 5 + 16 b_0 - 3 d_0.$$
Using \eqref{eq:nde} and \eqref{eq:eps34}, one gets 
$$h^0(N_0) = (n_0 + 1) n_0 + 3 k_0 - 2 b_0 -5.$$

As for  Step $(c)$, consider the exact sequence \eqref{eq:eps7}. $A_0$ and $B_0$ are uniquely determined on 
$\FF_0$. As in the proof of Theorem \ref{thm:parcount}, to compute $\dim(\mathcal Y_0)$ we have to add up the quantities   
$\tau_0$, that is the number of parameters counting isomorphism classes of projective bundles $\Pp(\mathcal E_0)$, and 
the dimension of the full orbit of $X_0  \subset \Pp^{n_0}$ under the action of $PGL(n_0+1, \mathbb C)$.

From  \eqref{eq:dimExt1eps} we get 
$$\tau_0 = \left\{
\begin{array}{ccc}
0 & & {\rm for} \; b_0 + \frac{e}{2} < k_0 <  \frac{3b_0+2}{2}\\
4k_0 - 6 b_0 - 3 & &    {\rm for} \; \frac{3b_0+2}{2} \le k_{0} < 2b_0-e,
\end{array}
\right.$$(cf.\,the proof of Theorem \ref{thm:parcount}).

The dimension of the orbit of $X_0$ is given by 
$$\dim(PGL(n_0+1, \mathbb C)) - \dim (G_{X_0}) = n_0(n_0 + 2) - h^0(T_{X_0}),$$
where $G_{X_0} \subset PGL(n_0+1, \mathbb C)$ 
is the projective stabilizer. In particular, 
$$\dim(\mathcal Y_0) = \tau_0 + n_0(n_0+2) - \dim(G_{X_0}).$$
As in the proof of Theorem \ref{thm:parcount}, 
one obviously has 
$$\dim(G_{X_0}) \leq \dim(Aut(X_0)) = \dim (Aut(\FF_0)) + \dim(Aut_{\FF_0} (\Pp({\mathcal E}_0)),$$where $Aut(X_0)$ denotes the algebraic group of abstract automorphisms of $X_0$ whereas $Aut_{\FF_0} (\Pp({\mathcal E}_0))$ the group of automorphisms
of $\Pp({\mathcal E}_0)$ fixing the base (cf. e.g. \cite{ma}). 

From \eqref{eq:400}, we have $ \dim (Aut(\FF_0)) = 6$ (cf. also \cite[Lemma 10]{ma}).

For $\dim(Aut_{\FF_0} (\Pp({\mathcal E}_0)))$,  from 
\eqref{eq:aut01} one gets 
$$\dim(Aut_{\FF_0} (\Pp({\mathcal E}_0)))  = \left\{
\begin{array}{ccl}
6b_0-4k_0 + 3 & & {\rm for} \;\; b_0 + \frac{e}{2} < k_0 < \frac{3b_0+2}{2}\\
0 & &  {\rm for} \;\; \frac{3b_0+2}{2} < k_0  < 2b_0  -e \;\; {\rm and} \;\; {\mathcal E}_{0} \; {\rm general}
\end{array}
\right.$$

In all cases, one gets 
$$\dim(\mathcal Y_0) \ge n_0(n_0 + 2) + 4 k_0 - 6b_0 - 9.$$
For Step $(d)$, we recall \eqref{eq:expdimeps}. So we have   
$$n_0(n_0 +1) + 3 k_0 - 2b_0 - 5 = \dim(\mathcal X_0) \ge \dim(\mathcal Y_0) \ge n_0(n_0 + 2) + 4 k_0 - 6b_0 - 9.$$
Observe 
that the left and right most sides of the previous inequalities are equal: indeed
$n_0(n_0 +1) + 3 k_0 - 2b_0 - 5 - (n_0(n_0 + 2) + 4 k_0 - 6b_0 - 9) = 4b_0 + 4 -  k_0 - n_0= 0$ as it follows from 
\eqref{eq:neps}. Thus $\dim(\mathcal X_0) = \dim(\mathcal Y_0)$ which concludes the proof. 
\end{proof}

%
%

\subsection{The components $\mathcal X_e$ and $\mathcal X_{\epsilon}$ coincide}\label{S:genpointeps}

\begin{theo}\label{thm:puntogenerico} With Assumptions \ref{ass:AB2}, one has $\mathcal X_e = \mathcal X_{\epsilon}$.  
\end{theo}

\begin{proof}  Notice that, from the proof of Lemma \ref{lem:AB1epsAB2eps}, Assumptions \ref{ass:AB2} are equivalent to conditions in \eqref{eq:newboundseps} which are exactly the values for which $\mathcal X_{\epsilon}$ has been constructed.

Recall that $\mathcal X_e$ and $\mathcal X_{\epsilon}$ have the same dimension (cf.\,Remark \ref{rem:dimeps}) and are both components of the same Hilbert scheme $\mathcal H_3^{d_e,n_e}$ as in \S\;\ref{S:hilbX}, since $X_e$ and $X_{\epsilon}$ have the same Hilbert polinomial (cf.\,\S\,\ref{S:5.3new}). From Theorems \ref{thm:HilbertFe} and \ref{thm:parcount}, we furthermore have that $[X_e] \in \mathcal Y_e$ general is a smooth point for $\mathcal X_e$ and similarly, Proposition \ref{prop:5.5eps} states that $[X_{\epsilon}] \in \mathcal X_{\epsilon}$ general is a smooth point too. Thus, by smoothness and the fact that $\dim(\mathcal X_{\epsilon}) = \dim(\mathcal X_e)$, 
to prove the theorem it will be enough to exhibit a flat, embedded (in $\Pp^{n_e}$) degeneration of $X_{\epsilon}$ to $X_e$ which is entirely contained in the smooth locus of $\mathcal X_{\epsilon}$; in other words, we need to show that there exist a flat family
\[
\xymatrix@C=0mm@R=5mm{\mathfrak F \ar[d]_\pi & \subset & \Pp^{n_e} \times \Delta \ar[dll]^{\text{pr}_2} \\
\Delta}
\]where $\Delta$ is a smooth, irreducible affine curve, $\text{pr}_2$ is the projection onto the second factor, $\mathfrak F \subset \Pp^{n_e} \times \Delta$ is a closed subscheme of relative dimension three, $\pi$ is the restriction to it of $\text{pr}_2$, which is proper, flat and such that $\pi^{-1}(t) := \mathfrak F_t \cong X_{\epsilon}$, for $t \neq 0$, and $\pi^{-1}(0) = \mathfrak F_0 \cong X_{e}$, and $\Delta$ maps to 
an (affine) irreducible curve in $\mathcal H_3^{d_e,n_e}$ (which,  by abuse of notation, we will always denote by $\Delta$) connecting $[X_{\epsilon}]$ with $[X_e]$ and such that $\Delta \subset (\mathcal X_{\epsilon})_{sm}$, the smooth locus 
of $\mathcal X_{\epsilon}$.

To exhibit this degeneration, recall that $X_e$ and $X_{\epsilon}$ are respectively  determined by the pairs $(\FF_e, \Ee)$ and $(\FF_{\epsilon}, \mathcal E_{\epsilon})$ (cf.\,Prop.\,\ref{prop:X} and  \eqref{eq:Xeps}). According to what was proved in the previous sections, when $\dim({\rm Ext^1}(B_e, A_e)) >0$ it is clear that the bundle 
$\mathcal E_e$ flatly degenerates  (or specializes, in the sense of \cite[p.\,126]{Ball}) inside the vector space ${\rm Ext^1}(B_e, A_e)$ to the decomposable bundle $A_e \oplus B_e$; when otherwise  $\dim({\rm Ext^1}(B_e, A_e)) =0$ one simply has $\mathcal E_e = A_e \oplus B_e$.  The same occurs for  bundles in  ${\rm Ext^1}(B_{\epsilon}, A_{\epsilon})$ on $\FF_{\epsilon}$.

Denote by $D_e$ (respectively $D_{\epsilon}$) the {\em decomposable} scroll determined by the pair \linebreak $(\Pp(A_e \oplus B_e ) , \Oc_{\Pp(A_e \oplus B_e)} (1))$ (respectively  $(\Pp(A_{\epsilon} \oplus B_{\epsilon} ), \Oc_{\Pp(A_{\epsilon} \oplus B_{\epsilon})} (1))$).

From the proofs of Theorem \ref{thm:HilbertFe} and Proposition \ref{prop:5.5eps}, $[X_e]$, $[D_e]$, $[X_{\epsilon}]$ and $[D_{\epsilon}]$ are all smooth points of the Hilbert scheme $\mathcal H_3^{d_e,n_e}$ and the flat (abstract) degenerations of general bundles in ${\rm Ext^1}(B_e, A_e)$ 
and in ${\rm Ext^1}(B_{\epsilon}, A_{\epsilon})$ to the decomposable ones $A_e \oplus B_e$ and $A_{\epsilon} \oplus B_{\epsilon}$, respectively,  clearly give rise to flat degenerations, embedded in $\Pp^{n_e}$, of $X_e$ to $D_e$ and of $X_{\epsilon}$ to $D_{\epsilon}$, which are contained in the smooth locus of $\mathcal X_e$ and $\mathcal X_{\epsilon}$, respectively. The assertions follow from the fact that, since all the bundles involved are very-ample and with no higher cohomology (cf.\;previous sections), 
the corresponding threefold scrolls  are smooth with non-special normal bundles in $\Pp^{n_e}$.

It is therefore enough to show that there exists a flat, embedded degeneration of $D_{\epsilon}$ to $D_e$ which is entirely contained in the smooth locus of $\mathcal X_{\epsilon}$; if this is the case, 
by smoothness at each step and by $\dim(\mathcal X_e) = \dim(\mathcal X_{\epsilon})$, we must have $\mathcal X_e = \mathcal X_{\epsilon}$ as desired.

Now, the decomposable scroll $D_e$ has two disjoint sections, say $S^{\alpha_e}$ and $S^{\beta_e}$, where \linebreak
$\alpha_e := \deg( S^{\alpha_e} ) = \deg (A_e) = 8b_e - 4 k_e - 12e$ and $\beta_e:=  \deg( S^{\beta_e} ) = \deg(B_e) = 2k_e - 2b_e +3e$ (cf.\,\eqref{eq:al-be3}), which correspond to the two quotients $A_e \oplus B_e \to\!\!\!\to A_e$ and $A_e \oplus B_e \to \!\!\!\to B_e$ respectively. Precisely, 
$S^{\alpha_e}$ (respectively $S^{\beta_e}$) is given by the embedding of $\FF_e$ via the very-ample linear system $|A_e|$ (respectively  $|B_e|$); from Lemma \ref{lemma:comEAB} and the non-speciality of both $A_e$ and $B_e$, the projective linear spans of such surfaces $\langle S^{\alpha_e} \rangle \cong \Pp^{\ell_e}$ and $\langle S^{\beta_e} \rangle \cong \Pp^{r_e}$, where $\ell_e:= h^0(A_e) -1 = 6b_e - 3k_e - 9e +2$ and $r_e:= h^0(B_e)-1 = 2k_e - 2b_e + 3e+1 = \beta_e+1$, are skew, spanning the whole $\Pp^{n_e}$, and $D_e$ turns out to be the join of these two surfaces.  

Similarly $D_{\epsilon}$ is the joint in $\Pp^{n_e}$ of two smooth, rational surfaces $S^{\alpha_{\epsilon}}$ and $S^{\beta_{\epsilon}}$, with 
$S^{\alpha_{\epsilon}}$  and $S^{\beta_{\epsilon}}$ respectively given by the embedding of $\FF_{\epsilon}$ via $|A_{\epsilon}|$ and $|B_{\epsilon}|$,  where $\alpha_{\epsilon} = \deg( S^{\alpha_{\epsilon}}) = \deg(A_{\epsilon}) =  \alpha_e$ and $\beta_{\epsilon} = \deg( S^{\beta_{\epsilon}}) = \deg(B_{\epsilon}) = \beta_e$, the last equalities following from \eqref{eq:eps34}, \eqref{eq:eps5}, \eqref{eq:eps6}. As above, these two surfaces are (disjoint) sections of $D_{\epsilon}$, whose linear spans  $\langle S^{\alpha_{\epsilon}} \rangle \cong \Pp^{\ell_e}$ and $\langle S^{\beta_{\epsilon}}\rangle \cong \Pp^{r_e}$ are skew, spanning the whole $\Pp^{n_e}$ (all the assertions follow from \eqref{eq:eps34}, \eqref{eq:eps5}, \eqref{eq:eps6},  the non--speciality of $A_{\epsilon}$ and of $B_{\epsilon}$ and the fact that the bundle is decomposable).

Since $|B_{\epsilon}|$ (respectively $|B_e|$) is very-ample and unisecant to the fibers of $\FF_{\epsilon}$ (respectively of 
$\FF_e$), the image surface $S^{\beta_{\epsilon}}$ (respectively $S^{\beta_e}$) is a smooth, rational normal scroll inside 
$\Pp^{r_e}$. If we denote by $\mathcal H_2^{\beta_e,r_e}$ the Hilbert scheme of rational normal  scrolls of degree $\beta_e $ in $\Pp^{r_e}$, it is well-known that  it is irreducible, smooth at points corresponding to smooth scrolls and that its general point is given by {\em balanced} scrolls, i.e. those arising from $\FF_{\epsilon}$. In particular, there is a flat  degeneration  of $S^{\beta_{\epsilon}}$ to $S^{\beta_e}$, embedded in $\Pp^{r_e}$, represented by an affine curve denoted by $\Delta$, connecting the Hilbert point $[S^{\beta_{\epsilon}}]$ to $[S^{\beta_e}]$ and which is entirely contained in the smooth locus of  $\mathcal H_2^{\beta_e,r_e}$ (cf.\,e.g.\,\cite[Def.\,2.15,\,Rem.\,3.9]{calabri-ciliberto-flamini-miranda:non-special} and \cite[Lemma 3]{clm} or read details below for the case with $|A_{\epsilon}|$ and $|A_e|$).

Similarly, denoting by $\mathcal H_2^{\alpha_e,\ell_e}$ the Hilbert scheme of closed subschemes of $\Pp^{\ell_e}$ having the same Hilbert polynomial as $S^{\alpha_{\epsilon}}$ (equivalently $S^{\alpha_e}$), one can easily show that there exists a flat embedded (in $\Pp^{\ell_e}$) degeneration  of $S^{\alpha_{\epsilon}}$ to $S^{\alpha_e}$, represented by the same $\Delta$ as above, connecting $[S^{\alpha_{\epsilon}}]$ to $[S^{\alpha_e}]$ and which is entirely contained in the smooth locus of a component of $\mathcal H_2^{\alpha_e,\ell_e}$.

To do this, for simplicity we focus on the case $e$ even, i.e. $\epsilon = 0$, since for $e$ odd the arguments hold almost verbatim. Take therefore for a moment $e = 2k \geq 2$; the non-trivial extension $0 \to \Oc_{\Pp^1} (-k) \to \Oc_{\Pp^1} \oplus \Oc_{\Pp^1} \to \Oc_{\Pp^1} (k) \to 0$ over $\Pp^1$ gives rise to a line of the vector space ${\rm Ext^1}(\Oc_{\Pp^1} (k), \Oc_{\Pp^1} (-k)) \cong {\rm Ext^1}(\Oc_{\Pp^1}, \Oc_{\Pp^1} (-e))$, which can be identified with a $1$-dimensional, affine base scheme $\Delta$ of a flat degeneration (or specialization, in the sense of \cite[p.\,126]{Ball}) of the bundle $\Oc_{\Pp^1} \oplus \Oc_{\Pp^1}$ to $\Oc_{\Pp^1} (k) \oplus \Oc_{\Pp^1} (-k)$ over $\Pp^1$, and so of 
$\FF_0$ to $\Pp(\Oc_{\Pp^1} (k) \oplus \Oc_{\Pp^1} (-k)) \cong \FF_e$. 

Since $\FF_0$ and $\FF_{2k}$ are endowed with very-ample line bundles $A_0$ and $A_{2k}$, respectively, of same degree and same projective dimension, it is a standard procedure to identify $\Delta$ as above with also the base scheme of a flat, embedded 
(in $\Pp^{\ell_0}$) degeneration of smooth, rational surfaces $S^{\alpha_0}$ to $S^{\alpha_{2k}}$ (cf.\;e.g.\,\cite{FrMo}\,and\,\cite[Constr.\;3.6,\,3.7]{calabri-ciliberto-flamini-miranda:non-special} for procedures in even more degenerate situations). Briefly, one takes the trivial family $\mathcal T := \FF_0 \times \Delta \stackrel{\text{pr}_2}{\longrightarrow} \Delta$, which is also endowed with a relative line bundle $\mathcal A$ resticting to $A_0$ on any $\text{pr}_2$-fiber. One then performs standard operations involving: (1) blowing-ups and blowing-downs in
the central fiber of $\mathcal T$, and (2) twisting $\mathcal A$ by components of the central fiber. Doing this, one gets a 
birational modification of the (original) central fiber $(\mathcal T_0, \mathcal A_{|_{\mathcal T_0}}) = (\FF_0, A_0)$ and a (no more trivial) proper, flat family $\mathcal T' \stackrel{\pi'}{\longrightarrow} \Delta$, together with a relative line bundle $\mathcal A' \to \mathcal T'$ s.t.:  the total space $\mathcal T'$ is smooth, \linebreak if $\mathcal T'_t := \pi'^{-1}(t)$ for $t \in \Delta$, then $h^0(\mathcal A'_{|_{\mathcal T'_t}}) = \alpha_0+1$, for any $t \in \Delta$, $(\mathcal T'_t, \mathcal A'_{|_{\mathcal T'_t}}) = (\mathcal T_t, \mathcal A_{|_{\mathcal T_t}}) = (\FF_0, A_0) \cong S^{\alpha_0} \subset \Pp^{\ell_0}$, for $t \neq 0$, whereas $(\mathcal T'_0, \mathcal A'_{|_{\mathcal T'_0}}) \cong  (\FF_{2k}, A_{2k}) \cong S^{\alpha_{2k}} \subset \Pp^{\ell_0}$ (cf.\,e.g.\,\cite{FrMo} and \cite{calabri-ciliberto-flamini-miranda:non-special} for full details). This means that $\Delta$ can be identified as an affine curve, always denoted by $\Delta$, in $\mathcal H_2^{\alpha_e,\ell_e}$ with the desired properties (the fact that $\Delta$ is entirely contained in the smooth locus of a component of  $\mathcal H_2^{\alpha_e,\ell_e}$ follows from the fact that the normal bundles in $\Pp^{\ell_0}$ of both $S^{\alpha_0}$ and $S^{\alpha_{2k}}$ are non-special, as it follows from the Euler sequence restricted to them).

Turning back to the general case with any $e$ and $\epsilon = 0,1$, it is then clear that for $t \in \Delta \setminus \{0\}$ approaching to $0$ we have ''simultaneous'' specializations of $S^{\alpha_{\epsilon}}$ to $ S^{\alpha_e}$ in $\Pp^{\ell_e}$ and of 
$S^{\beta_{\epsilon}}$ to $ S^{\beta_e}$ in $\Pp^{r_e}$ and so of their respective join in $\Pp^{n_e}$. Formally one applies the same procedures explained above to both pairs $(\FF_{\epsilon}, A_{\epsilon})$ and 
$(\FF_{\epsilon}, B_{\epsilon})$ and so also to $(\FF_{\epsilon}, A_{\epsilon} \oplus B_{\epsilon})$; in this way $\Delta$ can be 
identified with the base scheme of the desired flat family $\mathfrak F \stackrel{\pi}{\to} \Delta$ as in the beginning of the proof, whose general fiber is given by $(\FF_{\epsilon}, A_{\epsilon} \oplus B_{\epsilon}) \cong D_{\epsilon}$ and whose 
central fiber is $(\FF_e, A_e \oplus B_e) = D_e$ (notice that flatness of $\mathfrak F$ over $\Delta$ follows from the facts that $\Delta$ is integral and that all the fibers have the same Hilbert polynomial as in \eqref{eq:numpol}, cf.\cite[Prop.\,4.2.1\,(ii)]{Se}). Very-ampleness and non-speciality of $A_e \oplus B_e$ imply that $D_e$ and $D_{\epsilon}$ are smooth, non-special threefold scrolls in $\Pp^{n_e}$ with $h^1(N_{D_{\epsilon}/\Pp^{n_e}}) = h^1 (N_{D_e/\Pp^{n_e}}) = 0$ (cf.\,proofs of Claim \ref{cl:clar} and of 
Prop.\,\ref{prop:5.5eps}), i.e. the curve $\Delta$ is entirely contained in the smooth locus of $\mathcal H_3^{d_e,n_e}$ and so of $\mathcal X_{\epsilon}$, being one irreducible component of the Hilbert scheme. This forces $\mathcal X_{\epsilon} = \mathcal X_e$ as desired. 
\end{proof}

\begin{rem}\label{rem:ballico} {\normalfont The proof of Theorem \ref{thm:puntogenerico} can be interpreted as a projective-geometry counterpart 
of (abstract) specializations of rank-five vector bundles over $\Pp^1$ as in \cite[Prop.\,2.3]{Ball}. 
Applying the direct image functors $R^j {\pi_e}_*$  to the exact sequence \eqref{eq:al-be} gives the following exact sequence of bundles on $\Pp^1$
\begin{align}
  0 \to & {\pi_e}_*(A_e) \cong Sym^2(\Oc_{\Pp^1} \oplus \Oc_{\Pp^1}(-e)) \otimes \Oc_{\Pp^1}(2b_e-k_e-2e)  \to   {\pi_e}_*({\mathcal E}_e)  \nonumber  \\  
  \nonumber \to & {\pi_e}_*(B_e) \cong (\Oc_{\Pp^1} \oplus \Oc_{\Pp^1}(-e)) \otimes \Oc_{\Pp^1}(k_e-b_e+2e)  \to  0,  \nonumber 
\end{align}that is
\begin{align}
0 \to & \Oc_{\Pp^1} (2b_e-k_e-2e) \oplus \Oc_{\Pp^1}(2b_e-k_e-3e) \oplus  \Oc_{\Pp^1}(2b_e-k_e-4e)  \to   {\pi_e}_*({\mathcal E}_e) \nonumber \\
\to & \Oc_{\Pp^1} (k_e-b_e+2e) \oplus \Oc_{\Pp^1}(k_e-b_e+e)  \to  0. \nonumber 
\end{align}
Thus the push-forward via ${\pi_e}_*$ defines a natural map 
$${\rm Ext}^1(B_e, A_e) \stackrel{\Pi_e}{\longrightarrow} {\rm Ext}^1({\pi_e}_*(B_e), {\pi_e}_*(A_e)), \;\;{\rm s.t.}\; \; \Pi_e(\mathcal E_e) := {\pi_e}_*({\mathcal E}_e).$$Now ${\pi_e}_*({\mathcal E}_e)$ is a rank-five vector bundle on $\Pp^1$ with 
$\delta_e:= \deg({\pi_e}_*({\mathcal E}_e)) = 4b_e-k_e-6e$ so ${\pi_e}_*({\mathcal E}_e) = \bigoplus_{i=1}^5 \Oc_{\Pp^1}(\alpha_i)$, 
for some $\alpha_i \in \mathbb{Z}$ with $\Sigma_{i=1}^5 \alpha_i = 4b_e-k_e-6e$.

Similarly, from \eqref{eq:eps7} one gets 
\begin{align}
0 \to & {\pi_{\epsilon}}_*(A_{\epsilon}) \cong Sym^2(\Oc_{\Pp^1} \oplus \Oc_{\Pp^1}(-\epsilon)) \otimes \Oc_{\Pp^1}(2b_{\epsilon}-k_{\epsilon}-2e)  \to   {\pi_{\epsilon}}_*({\mathcal E}_{\epsilon})\nonumber \\
\to & {\pi_{\epsilon}}_*(B_{\epsilon}) \cong (\Oc_{\Pp^1} \oplus \Oc_{\Pp^1}(-\epsilon)) \otimes \Oc_{\Pp^1}(k_{\epsilon}-b_{\epsilon}+2e)  \to  0 \nonumber
\end{align}
which reads also
\begin{align}
0 \to & \Oc_{\Pp^1} (2b_{\epsilon}-k_{\epsilon}-2\epsilon) \oplus \Oc_{\Pp^1}(2b_{\epsilon}-k_{\epsilon}-3\epsilon) \oplus  \Oc_{\Pp^1}(2b_{\epsilon}-k_{\epsilon}-4\epsilon)  \to   {\pi_{\epsilon}}_*({\mathcal E}_{\epsilon}) \nonumber\\
\to & \Oc_{\Pp^1} (k_{\epsilon}-b_{\epsilon}+2\epsilon) \oplus \Oc_{\Pp^1}(k_{\epsilon}-b_{\epsilon}+\epsilon) \to  0.\nonumber
\end{align}As above ${\pi_{\epsilon}}_*({\mathcal E}_{\epsilon})$ is decomposable, of rank five 
on $\Pp^1$, with $\deg({\pi_{\epsilon}}_*({\mathcal E}_{\epsilon})) = 4 b_{\epsilon} - k_{\epsilon} - 6 \epsilon$. From \eqref{eq:eps34} one has  $4 b_{\epsilon} - k_{\epsilon} - 6 \epsilon = 4b_e-k_e-6e$, i.e. $\deg({\pi_{\epsilon}}_*({\mathcal E}_{\epsilon})) = \deg({\pi_e}_*({\mathcal E}_e)) = \delta_e$. 

It is clear that, inside ${\rm Ext}^1({\pi_e}_*(B_e), {\pi_e}_*(A_e))$, the bundle ${\pi_e}_*({\mathcal E}_e)$ 
flatly degenerates (or is equal) to the bundle 
{\small $$\mathcal T_e := \Oc_{\Pp^1} (2b_e-k_e-2e) \oplus \Oc_{\Pp^1}(2b_e-k_e-3e) \oplus  \Oc_{\Pp^1}(2b_e-k_e-4e) \oplus \Oc_{\Pp^1} (k_e-b_e+2e) \oplus \Oc_{\Pp^1}(k_e-b_e+e).$$}For simplicitly, put $$\xi'_1 := 2b_e-k_e-2e , \; \xi'_2 := 2b_e-k_e-3e, \; \xi'_3 := 2b_e-k_e-4e$$and 
$$\eta_1':= k_e-b_e+2e, \; \eta_2':= k_e-b_e+e.$$Similarly, inside ${\rm Ext}^1({\pi_{\epsilon}}_*(B_{\epsilon}), {\pi_{\epsilon}}_*(A_{\epsilon}))$, the vector bundle 
${\pi_{\epsilon}}_*({\mathcal E}_{\epsilon})$ flatly degenerates (or is equal) to 
{\small $$\mathcal T_{\epsilon} := \Oc_{\Pp^1} (2b_{\epsilon}-k_{\epsilon}-2\epsilon) \oplus \Oc_{\Pp^1}(2b_{\epsilon}-k_{\epsilon}-3\epsilon) \oplus  \Oc_{\Pp^1}(2b_{\epsilon}-k_{\epsilon}-4\epsilon)  \oplus \Oc_{\Pp^1} (k_{\epsilon}-b_{\epsilon}+2\epsilon) \oplus \Oc_{\Pp^1}(k_{\epsilon}-b_{\epsilon}+\epsilon).$$
}Using \eqref{eq:eps34}, the latter reads  
{\small
\begin{eqnarray}
\begin{array}{ccl}
\mathcal T_{\epsilon}  & = & \Oc_{\Pp^1} (2b_e-k_e-3e + \epsilon) \oplus \Oc_{\Pp^1}(2b_e-k_e-3e) \oplus  \Oc_{\Pp^1}(2b_e-k_e-3e-\epsilon)  \nonumber\\
 & & \oplus \Oc_{\Pp^1} (k_e-b_e+\frac{(3e+\epsilon)}{2}) \oplus \Oc_{\Pp^1}(k_e-b_e+\frac{(3e-\epsilon)}{2}).\nonumber
\end{array}
\end{eqnarray} 
}As above, for simplicity, put 
$$\xi_1 := 2b_e-k_e-3e + \epsilon , \; \xi_2 := 2b_e-k_e-3e, \; \xi_3 := 2b_e-k_e-3e-\epsilon$$and 
$$\eta_1 := k_e-b_e+ \frac{(3e+\epsilon)}{2}, \; \eta_2 := k_e-b_e+\frac{(3e-\epsilon)}{2}.$$
By \cite[Prop 2.3]{Ball}, one deduces that $\mathcal T_e$ is a flat specialization of $\mathcal T_{\epsilon}$; indeed, they have same rank and same degree but the latter is more balanced since, for any $1 \le i \le 2$:  
$$\{0,1\} \ni \epsilon = \eta_2-\eta_1 = \xi_{i+1} - \xi_i \;\; {\rm whereas} \;\; 2 \le e = \eta'_2-\eta'_1 = \xi'_{i+1} - \xi'_i.$$

}
\end{rem}

\noindent

%
%
 
\section{Examples}\label{Examples}

We give some examples of Hilbert schemes of threefold scrolls over  $\FF_e$, with $e$  both even and odd. We use notation and assumptions as in the previous sections. 

\vskip 10pt


\noindent
(1) Take $ e=2$, $b_2 = 11$, $k_2=11$, which are compatible with \eqref{eq:newbounds}. Consider vector bundles 
$\mathcal E_2$ over $\FF_2$ fitting in
$$ 0 \to A_2 = 2 C_2 + 7 f \to \mathcal E_2 \to B_2 = C_2 + 4 f \to 0.$$More precisely, 
since ${\rm Ext}^1(B_2,A_2) \cong H^1(C_2 + 3 f) = (0)$, then $\mathcal E_2 = (2 C_2 + 7 f) \oplus (C_2 + 4 f)$. 
One has $h^0(\mathcal E_2) = 26$, $h^i(\mathcal E_2) = 0$, for $i \ge 1$, and $d_2 = \deg(\mathcal E_2) = 37$. 

For $X_2 \subset \Pp^{25}$ we know that 
$h^1(N_2)=h^1(N_{X_2/\Pp^{25}}) = 0$ (cf. the proof of Claim \ref{cl:clar}). Then $[X_2] \in \mathcal X_2$ is a smooth point, where 
$\mathcal X_2$ is generically smooth of dimension $662$. 

From \eqref{eq:eps34}, on $\FF_0 \cong \Pp^1 \times \Pp^1$ we take vector bundles $\mathcal E_0$ fitting in 
$$ 0 \to A_0 = 2 C_0 + 5 f \to \mathcal E_0 \to B_0 = C_0 + 3 f \to 0,$$compatible with \eqref{eq:newboundseps}. 
As above, since ${\rm Ext}^1(B_0,A_0) \cong H^1(C_0 + 2 f) = (0)$, then $\mathcal E_0 = (2 C_0 + 5 f) \oplus (C_0 + 3 f)$. 
$\mathcal E_0$ has the same degree and the same cohomology as that of $\mathcal E_2$. 
Let $X_0 \subset \Pp^{25}$ be the associated threefold scroll. From Proposition \ref{prop:5.5eps} and Theorem \ref{thm:puntogenerico}, 
$[X_0] \in \mathcal X_2$ is the general point.

In terms of vector bundles as in Remark \ref{rem:ballico}, notice that up to a descending reorder of the summands we have  
$${\pi_2}_*(\mathcal E_2) = \mathcal T_2 = \Oc_{\Pp^1}(7) \oplus \Oc_{\Pp^1}(5) \oplus \Oc_{\Pp^1}(4) \oplus \Oc_{\Pp^1}(3)\oplus \Oc_{\Pp^1}(2)$$and 
$${\pi_0}_*(\mathcal E_0) = \mathcal T_0= \Oc_{\Pp^1}(5)^{\oplus 3} \oplus \Oc_{\Pp^1}(3)^{\oplus 2}$$so 
${\pi_2}_*(\mathcal E_2) = \mathcal T_2$ is a flat specialization of $ {\pi_0}_*(\mathcal E_0) = \mathcal T_0$ (\cite[Prop.\,2.3]{Ball}).

\bigskip 

\noindent
(2) From \eqref{eq:newbounds}, take $ e=3$, $b_3 = 15$, $k_3=15$. Consider vector bundles 
$\mathcal E_3$ over $\FF_3$  fitting in $$ 0 \to A_3 = 2 C_3 + 8 f \to \mathcal E_3 \to B_3 = C_3 + 7 f \to 0.$$ 
Since $15 = k_3 < 2b_3  - 4e = 18$, from the first line of \eqref{eq:h1A}, $h^1(A_3) = 0$ so the same holds for any 
$\mathcal E_3 \in {\rm Ext}^1(B_3,A_3) \cong H^1(C_3 + f) \cong  \mathbb{C}$ (cf. Corollary \ref{cor:van}). All 
$\mathcal E_3$'s have degree $d_3 = 47$, $h^0(\mathcal E_3) = 32$ and no higher cohomology. Moreover, any $\mathcal E_3$ corresponding to a non-zero vector in 
$ {\rm Ext}^1(B_3,A_3) $ flatly degenerates inside this vector space to the trivial bundle $\mathcal T_3 := A_3 \oplus B_3$.

From \eqref{eq:eps34}, on $\FF_1$ we correspondingly take 
$$ 0 \to A_1 = 2 C_1+ 6 f\to  \mathcal E_1 \to B_1 = C_1+ 6 f \to 0.$$Now  
${\rm Ext}^1(B_1,A_1) \cong H^1(C_1) \cong (0)$ and thus  $\mathcal E_1 =  A_1 \oplus B_1$ is the unique bundle. 
From the proof of Theorem \ref{thm:puntogenerico}, these all correspond to smooth points of the Hilbert scheme $\mathcal H_3^{27,31}$, in particular contained in the same irreducible component $\mathcal X_3$, which is generically smooth.

In terms of vector bundles on $\Pp^1$, we have that $${\pi_3}_*(\mathcal T_3):= \Oc_{\Pp^1}(8) \oplus \Oc_{\Pp^1}(7) \oplus \Oc_{\Pp^1}(5) \oplus \Oc_{\Pp^1}(2)  \oplus \Oc_{\Pp^1}(4),$$which corresponds to the zero-vector 
of ${\rm Ext^1}({\pi_3}_*(B_3),{\pi_3}_*(A_3))$. Similarly,  
$${\pi_1}_*(\mathcal T_1) = \Oc_{\Pp^1}(6)^{\oplus 2} \oplus\Oc_{\Pp^1}(5)^{\oplus 2} \oplus \Oc_{\Pp^1}(4).$$
The bundle  ${\pi_1}_*(\mathcal T_1)$ degenerates to ${\pi_3}_*(\mathcal T_3)$ since it is more balanced than ${\pi_3}_*(\mathcal T_3)$ (apply \cite[Prop 2.3]{Ball}).

\bigskip 

\noindent
(3) Take $ e=4$, $b_4 = 18$, $k_4=18$. Consider vector bundles $\mathcal E_4$ over $\FF_4$  
fitting in $$ 0 \to A_4 = 2 C_4 + 10 f \to \mathcal E_4 \to B_4 = C_4 + 8 f \to 0.$$As above, 
${\rm Ext}^1(B_4,A_4) \cong \mathbb{C}$, all bundles have degree $d_4 = 58$ and are 
such that $h^i(\mathcal E_4) = 0$, for $i \ge 1$, and $h^0(\mathcal E_4) = 35$. The general element in ${\rm Ext}^1(B_4,A_4) $ flatly degenerates to 
the trivial one $\mathcal T_4 = A_4 \oplus B_4$.

On $\FF_0 \cong \Pp^1 \times \Pp^1$ consider bundles fitting in  
$$ 0 \to A_0 = 2 C_0 + 6 f \to \mathcal E_0 \to B_0 = C_0 + 6 f \to 0.$$Now  
${\rm Ext}^1(B_0,A_0) \cong H^1(C_0) = (0)$. Similarly as in (2), 
$${\pi_4}_*(\mathcal T_4) = \Oc_{\Pp^1}(10) \oplus \Oc_{\Pp^1}(8) \oplus \Oc_{\Pp^1}(6) \oplus \Oc_{\Pp^1}(4) \oplus \Oc_{\Pp^1}(2)$$corresponds to the 
zero-vector of ${\rm Ext^1}({\pi_4}_*(B_4), {\pi_4}_*(A_4))$ wheras
$${\pi_0}_*(\mathcal E_0) = \Oc^{\oplus 5}_{\Pp^1}(6)$$ flatly degenerates to ${\pi_4}_*(\mathcal T_4)$, since it is more balanced (apply e.g. \cite[Prop 2.3]{Ball}). As in example (2), we can conclude.


\end{document}